\documentclass[12pt,a4paper]{amsart}
\usepackage{amsmath,amssymb,amsthm,bm}
\usepackage{array,mathdots,blkarray,arydshln}
\usepackage{graphicx,xcolor,enumerate}
\usepackage{layout}
\usepackage{url}

\theoremstyle{plain}
\newtheorem{theorem}{Theorem}[section]
\newtheorem{lemma}[theorem]{Lemma}

\newtheorem{corollary}[theorem]{Corollary}

\theoremstyle{definition}
\newtheorem{remark}[theorem]{Remark}
\newtheorem{definition}[theorem]{Definition}

\newtheorem{example}[theorem]{Example}
\newtheorem*{acknowledgments}{Acknowledgments}

\numberwithin{equation}{section}

\newcommand{\bN}{\mathbb{N}}
\newcommand{\bZ}{\mathbb{Z}}

\newcommand{\bP}{\mathbb{P}}
\newcommand{\bG}{\mathbb{G}}
\newcommand{\ep}{\varepsilon}

\title[Two-dimensional quasi-monomial group actions]
{Rationality problem of two-dimensional quasi-monomial group actions}

\author[A. Hoshi]{Akinari Hoshi}
\address{Department of Mathematics, Niigata University, 
Niigata 950-2181, Japan}
\email{hoshi@math.sc.niigata-u.ac.jp}

\author[H. Kitayama]{Hidetaka Kitayama}
\address{Department of Mathematics, 
Wakayama University, Wakayama 640-8510, Japan}
\email{hkitayam@wakayama-u.ac.jp}

\thanks{
This work was partially supported by JSPS KAKENHI Grant Numbers 
19K03418, 19K03447, 20H00115. 
Parts of the work were finished when the authors were visiting the
National Center for Theoretic Sciences (Taipei), whose
support is gratefully acknowledged.}


\begin{document}
\begin{abstract}
The rationality problem of two-dimensional purely quasi-monomial actions 
was solved completely by Hoshi, Kang and Kitayama \cite{HKK}. 
As a generalization, we solve the rationality problem of two-dimensional 
quasi-monomial actions under the condition that the actions are defined 
within the base field. 
In order to prove the theorem, 
we give a brief review of the Severi-Brauer variety with 
some examples and rationality results. 
We also use a rationality criterion 
for conic bundles of $\bP^1$ over non-closed fields. 
\end{abstract}
\maketitle
\tableofcontents
%
%
\section{Introduction} \label{seInt}
%

Let $G$ be a finite group. 
A $G$-lattice $M$ is a finitely generated $\bZ[G]$-module which is 
$\bZ$-free as an abelian group, i.e.\ 
$M=\bigoplus_{1\le i\le n} \bZ\cdot x_i$ with a $\bZ[G]$-module structure. 
Let $K/k$ be a field extension such that $G$ acts on $K$ with $K^G=k$. 
Consider a short exact sequence of $\bZ[G]$-modules 
$\alpha: 1\to K^{\times}\to M_\alpha \to M\to 0$ 
where $M$ is a $G$-lattice and $K^{\times}$ is regarded as a 
$\bZ[G]$-module through the $G$-action on $K$. 
The $\bZ[G]$-module structure (written multiplicatively) of 
$M_\alpha$ may be described as follows: 
For each $x_j\in M$ ($1\le j\le n$), 
take a pre-image $u_j$ of $x_j$. 
As an abelian group, $M_\alpha$ is the direct product of $K^{\times}$ 
and $\langle u_1, \ldots, u_n \rangle$. 
If $\sigma \in G$ and 
$\sigma\cdot x_j=\sum_{1\le i\le n} a_{ij} x_i \in M$, 
we find that 
$\sigma\cdot u_j=c_j(\sigma) \cdot \prod_{1\le i\le n}u_i^{a_{ij}} 
\in M_\alpha$ for a unique $c_j(\sigma)\in K^{\times}$ determined 
by the group extension $\alpha$. 

Using the same idea, once a group extension $\alpha:1\to
K^{\times}\to M_\alpha \to M\to 0$ is given, we may define 
an action of $G$ on the rational function field 
$K(x_1,\ldots,x_n)$ as follows: 
If $\sigma\cdot x_j=\sum_{1\le i\le n} a_{ij} x_i \in M$, 
then define $\sigma\cdot x_j=c_j(\sigma)\prod_{1\le i\le n} x_i^{a_{ij}} 
\in K(x_1,\ldots,x_n)$ and $\sigma\cdot \alpha =\sigma(\alpha)$ 
for $\alpha\in K$ where $\sigma(\alpha)$ is the image of $\alpha$
under $\sigma$ via the prescribed action of $G$ on $K$. 
This action is well-defined (see Saltman \cite[page 538]{Sa1} for details). 
The field $K(x_1,\ldots,x_n)$ with such a $G$-action will be denoted 
by $K_\alpha(M)$ to emphasize the role of the extension $\alpha$; 
its fixed field is denoted as $K_\alpha(M)^G$. 
When $k=K$, the above group action is called {\it monomial group action} 
in Hajja and Kang \cite{HK1}. 
The elements of $k_\alpha(M)^G$ are called 
{\it twisted multiplicative field invariants} in 
Saltman \cite{Sa2} 
(see also Saltman \cite{Sa1} and Kang \cite[Definition 2.2]{Ka4}).

If the extension $\alpha$ splits, then we may take
$u_1,\ldots,u_n\in M_\alpha$ satisfying that $\sigma\cdot
u_j=\prod_{1\le i\le n} u_i^{a_{ij}}$. 
In this case, we will write $K_\alpha(M)$ and $K_\alpha(M)^G$ 
as $K(M)$ and $K(M)^G$ respectively 
(the subscript $\alpha$ is omitted because the extension $\alpha$ 
plays no important role). 
Again when $k=K$, 
the action of $G$ on $k(M)^G$ is called {\it purely monomial} 
in Hajja and Kang \cite{HK1}. 

The following gives an equivalent definition of the action of $G$ 
on $K_\alpha(M)$ above (see also Hoshi, Kang and Kitayama 
\cite[Section 1]{HKK} for some examples): 
\begin{definition}[{Hoshi, Kang and Kitayama \cite[Definition 1.1]{HKK}}]\label{defqm}
Let $k$ be a field and $K/k$ be a finite extension. 
Let $K(x_1,\ldots, x_n)$ be the rational function field over 
$K$ with $n$ variables $x_1,\ldots,x_n$. 
Let $G$ be a finite subgroup of ${\rm Aut}_k(K(x_1,\ldots, x_n))$. 
The $G$-action on $K(x_1,\ldots, x_n)$ is called {\it quasi-monomial} 
if it satisfies the following three conditions:\\
(i) $\sigma (K)\subset K$ for any $\sigma\in G$;\\
(ii) $K^G=k$ where $K^G$ is the fixed field under the action of $G$;\\
(iii) $\sigma(x_j)=c_j(\sigma)\prod_{i=1}^n x_i^{a_{i,j}}$
for some $c_j(\sigma)\in K\setminus\{0\}$
for any $\sigma\in G$ and $j=1,\ldots, n$,
where $[a_{i,j}]\in GL_n(\bZ)$.

The quasi-monomial action is called {\it purely quasi-monomial} 
if $c_j(\sigma)=1$ for any $\sigma\in G$, any $1\leq j\leq n$
in {\rm (iii)}.
\end{definition}

Let $G$ be a finite group acting on $K(x_1,\ldots,x_n)$
by quasi-monomial $k$-automorphisms. 
For the group homomorphism $\varphi: G \to GL_n(\bZ), \sigma\mapsto [a_{i,j}]_{1\leq i,j\leq n}$
where $[a_{i,j}]_{1\leq i,j\leq n}$ is given as in (iii) 
of Definition \ref{defqm}, 
define $N:= {\rm Ker} \varphi$.
By \cite[Proposition 1.12]{HKK},
we can take $y_1, \ldots, y_n$ such that
$K(x_1,\ldots,x_n)^N=K^N(y_1,\ldots,y_n)$ and
the action of $G/N$ on $K^N(y_1,\ldots,y_n)$ is again a quasi-monomial action.
Hence we may assume that $N=\{1\}$.
We identify $G$ with a finite subgroup of $GL_n(\bZ)$
without loss of generality hereafter.

Let $L$ be a finitely generated extension field of $k$. 
$L$ is called {\it $k$-rational} (or {\it rational over $k$}) 
if $L$ is purely transcendental over $k$, i.e. $L$ is
$k$-isomorphic to the quotient field of some polynomial ring over $k$.
$L$ is called {\it stably $k$-rational} if $L(y_1,\ldots,y_m)$ is 
$k$-rational for some $y_1,\ldots,y_m$ which are algebraically 
independent over $L$.
$L$ is called {\it $k$-unirational} if $L$ is $k$-isomorphic to a
subfield of some $k$-rational field. 
It is obvious that ``$k$-rational" $\Rightarrow$ ``stably $k$-rational" 
$\Rightarrow$ ``$k$-unirational". 
The L\"uroth problem asks, under what situation, the converse is true, 
i.e. ``$k$-unirational" $\Rightarrow$ ``$k$-rational". 
The L\"uroth problem is a famous problem in algebraic geometry. 
For a survey of it and related rationality problems 
(e.g. Noether's problem), see Manin and Tsfasman \cite{MT} 
and Swan \cite{Sw}.

Define $H=\{ \sigma \in G \ |\ \sigma(\alpha)=\alpha
\ {\rm for} \ {\rm any}\ \alpha \in K \}$.
When $H=\{1\}$, the fixed field $K(x_1,\ldots,x_n)^G$ is $K$-rational 
because $K(x_1,\ldots,x_n)^G\cap K=k$ and 
$K(x_1,\ldots,x_n)^G\otimes_k K=K(x_1,\ldots,x_n)$. 
In particular, $K(x_1,\ldots,x_n)^G$ is geometrically rational, i.e. 
$K(x_1,\ldots,x_n)^G\otimes_k \overline{k}=\overline{k}(x_1,\ldots,x_n)$ 
is $\overline{k}$-rational 
where $\overline{k}$ is an algebraic closure of $k$. 
We remark that when $H\neq\{1\}$, 
there exist $K(x_1,\ldots,x_n)^G$ which are 
not geometrically rational when $n\geq 4$, see 
Hoshi, Kang and Kitayama \cite[Theorem 6.2]{HKK}, 
Hoshi, Kang and Yamasaki \cite[Theorem 1.10]{HKY}. 

The aim of this paper is to investigate the $k$-rationality 
of the fixed field $K(x_1,x_2)^G$ under a quasi-monomial action of $G$.
The following theorems are known (see also Hoshi and Kitayama \cite{HK} 
for $3$-dimensional purely quasi-monomial actions):

\begin{theorem}[{Hoshi, Kang and Kitayama \cite[Proposition 1.13]{HKK}}]
\label{thHKK14a}
Let $k$ be a field and $K/k$ be a finite extension. \\
{\rm (1)} Let $G$ be a finite group acting on $K(x)$ by purely quasi-monomial $k$-automorphisms.
Then $K(x)^G$ is $k$-rational.\\
{\rm (2)} Let $G$ be a finite group acting on $K(x)$ by
quasi-monomial $k$-automorphisms. Then $K(x)^G$ is $k$-rational
except for the following case{\rm :} There is a normal subgroup
$N$ of $G$ such that {\rm (i)} $G/N=\langle \sigma \rangle \simeq
C_2$, {\rm (ii)} $K(x)^N=k(\alpha)(y)$ with $\alpha^2=a\in
K\setminus\{0\}$, $\sigma(\alpha)=-\alpha$ {\rm (}if ${\rm char}$ $k \ne
2${\rm )}, and $\alpha^2+\alpha=a\in K$, $\sigma(\alpha)=\alpha+1$
{\rm (}if ${\rm char}$  $k=2${\rm )}, {\rm (iii)} $\sigma\cdot y=b/y$
for some $b\in k\setminus\{0\}$.

For the exceptional case, $K(x)^G=k(\alpha) (y)^{G/N}$ is
$k$-rational if and only if the norm residue $2$-symbol
$(a,b)_{2,k}=0$ $($if ${\rm char}$ $k\ne 2$$)$, and $[a,b)_{2,k}=0$ $($if
${\rm char}$ $k=2$$)$.

Moreover, if $K(x)^G$ is not $k$-rational, then $k$ is an infinite
field, the Brauer group ${\rm Br}(k)$ is non-trivial, and $K(x)^G$
is not $k$-unirational.
\end{theorem}

\begin{theorem}[{Voskresenskii \cite[Theorem 2]{Vo}}]\label{thVos67}
Let $k$ be a field, $K/k$ be a finite extension and
$G$ be a finite subgroup of $GL_2(\bZ)$ acting on $K(x,y)$ by 
purely quasi-monomial $k$-automorphisms. 
Define $H=\{ \sigma \in G \ |\ \sigma(\alpha)=\alpha
\ {\rm for} \ {\rm any}\ \alpha \in K \}$.
If $H=\{1\}$, then $K(x,y)^G$ is $k$-rational.
\end{theorem}

\begin{theorem}[{Hajja \cite[Theorem 4]{Ha1}, \cite[Theorem]{Ha2}}]\label{thHaj8387}
Let $k$ be a field, $K/k$ be a finite extension and
$G$ be a finite subgroup of $GL_2(\bZ)$ acting on $K(x,y)$ by 
quasi-monomial $k$-automorphisms. 
Define $H=\{ \sigma \in G \ |\ \sigma(\alpha)=\alpha
\ {\rm for} \ {\rm any}\ \alpha \in K \}$.
If $H=G$, i.e. $K=k$, then $K(x,y)^G$ is $k$-rational.
\end{theorem}

\begin{theorem}[{Hoshi, Kang and Kitayama \cite[Theorem 1.14]{HKK}}]
\label{thHKY14}
Let $k$ be a field, $K/k$ be a finite extension and
$G$ be a finite subgroup of $GL_2(\bZ)$ acting on $K(x,y)$ by 
purely quasi-monomial $k$-automorphisms. 
Define $H=\{ \sigma \in G \ |\ \sigma(\alpha)=\alpha
\ {\rm for} \ {\rm any}\ \alpha \in K \}$,
and we assume that $H\neq \{1\}$.\\
$(1)$ $K(x,y)^G$ is $k$-rational if
${\rm char}$ $k=2$ or
$(G,H) \not\simeq (C_4,C_2)$, $(D_4,C_2)$.\\
$(2)$ If ${\rm char}$ $k \neq 2$ and $(G,H)\simeq (C_4,C_2)$,
then we may assume that $K=k(\sqrt{a})$ and
$G=\langle \sigma\rangle$ acts on $K(x,y)$ by
$\sigma : \sqrt{a}\mapsto -\sqrt{a}, \ x\mapsto y, \ y\mapsto 1/x$.
Then $K(x,y)^G$ is $k$-rational if and only if the norm residue $2$-symbol $(a,-1)_{2,k}=0$.\\
$(3)$
If ${\rm char}$ $k \neq 2$ and $(G,H)\simeq (D_4,C_2)$,
then we may assume that $K=k(\sqrt{a},\sqrt{b})$ and
$G=\langle \sigma, \tau \rangle$ acts on $K(x,y)$ by
$\sigma : \sqrt{a}\mapsto -\sqrt{a}, \ \sqrt{b}\mapsto \sqrt{b}, \ x\mapsto y, \ y\mapsto 1/x$, \
$\tau : \sqrt{a}\mapsto \sqrt{a}, \ \sqrt{b}\mapsto -\sqrt{b}, \ x\mapsto y, \ y\mapsto x$.
Then $K(x,y)^G$ is $k$-rational if and only if $(a,-b)_{2,k}=0$.

Moreover, if $K(x,y)^G$ is not $k$-rational, then $k$ is an
infinite field, the Brauer group ${\rm Br}(k)$ is non-trivial, and
$K(x,y)^G$ is not $k$-unirational.
\end{theorem}

The main theorem of this paper can be described as follows.
Let $G$ be a finite subgroup of $GL_2(\bZ)$.
The rationality is determined up to conjugacy in $GL_2(\bZ)$
since a conjugate of $G$ corresponds to some base change of
$K(x,y)$.
There are 13 $GL_2(\bZ)$-conjugacy classes of finite subgroups of $GL_2(\bZ)$.
See \cite{BBNWZ}, \cite{GAP}.
We use the following notation and the groups of representatives (cf. \cite{HKK}).\\

Cyclic groups:
\begin{align*}
C_1 &:=\{ I\}, &  C_2^{(1)} &:=\langle -I \rangle, &
C_2^{(2)} &:=\langle \lambda \rangle, &
C_2^{(3)} &:=\langle \tau \rangle, \\
C_3 &:= \langle \rho^2 \rangle, &
C_4 &:= \langle \sigma \rangle, &
C_6 &:= \langle \rho \rangle&
\end{align*}
where
\[
I = \begin{bmatrix} 1&0\\0&1 \end{bmatrix}, \
\lambda = \begin{bmatrix} 1&0\\0&-1 \end{bmatrix}, \
\tau = \begin{bmatrix} 0&1\\1&0 \end{bmatrix}, \
\sigma = \begin{bmatrix} 0&-1\\1&0 \end{bmatrix}, \
\rho = \begin{bmatrix} 1&-1\\1&0 \end{bmatrix}.
\]

Non-cyclic groups:
\begin{align*}
V_4^{(1)} &:=\langle \lambda, -I \rangle, &
V_4^{(2)} &:=\langle \tau, -I \rangle, &
S_3^{(1)} &:=\langle \rho^2, \tau \rangle, &
S_3^{(2)} &:=\langle \rho^2, -\tau \rangle, \\
D_4 &:=\langle \sigma, \tau \rangle, &
D_6 &:=\langle \rho, \tau \rangle.
\end{align*}
Note that
$\lambda^2=\tau^2=I$, $\sigma^2=\rho^3=-I$ and $\tau\sigma=\lambda$.

\vspace{10pt}
Let $G$ be one of the above $13$ group acting on $K(x,y)$
by quasi-monomial $k$-automorphisms. 
We define 
$H=\{ \sigma \in G \ |\ \sigma(\alpha)=\alpha
\ {\rm for} \ {\rm any}\ \alpha \in K \}$ as above. 
Then $H$ becomes a normal subgroup of $G$ 
and $K$ is a $(G/H)$-Galois extension of $k$. 
Normal subgroups $H$ of $G$ are given as follows:
\begin{align*}
&C_2^{(1)} : H=\{1\}, \ \langle -I \rangle, &
&C_2^{(2)} : H=\{1\}, \ \langle \lambda \rangle,\\
&C_2^{(3)} : H=\{1\}, \ \langle \tau \rangle, &
&C_3 : H=\{1\}, \ \langle \rho^2 \rangle,\\
&C_4 : H=\{1\}, \ \langle \sigma^2 \rangle, \ \langle \sigma \rangle, &
&C_6 : H=\{1\}, \ \langle -I \rangle, \ \langle \rho^2 \rangle, \ \langle\rho\rangle, 
\end{align*}
\begin{align*}
&V_4^{(1)} : H=\{1\}, \ \langle -I \rangle, \ \langle \lambda \rangle, \  \langle -\lambda \rangle, \ \langle \lambda,-I \rangle,\\
&V_4^{(2)} : H=\{1\}, \ \langle -I \rangle, \ \langle \tau \rangle, \ \langle -\tau \rangle, \ \langle \tau,-I \rangle,\\
&S_3^{(1)} : H=\{1\}, \ \langle \rho^2 \rangle, \ \langle \rho^2, \tau \rangle,\\
&S_3^{(2)} : H=\{1\}, \ \langle \rho^2 \rangle, \ \langle \rho^2, -\tau \rangle,\\
&D_4 : H=\{1\}, \ \langle -I \rangle, \ \langle -I,\lambda \rangle, \ \langle -I,\tau \rangle, \ \langle \sigma \rangle, \ \langle \sigma, \tau \rangle, \\
&D_6 : H=\{1\}, \ \langle -I \rangle, \ \langle \rho^2 \rangle, \
\langle \rho \rangle, \ \langle \rho^2,\tau \rangle, \
\langle \rho^2,-\tau \rangle, \langle \rho,\tau \rangle.
\end{align*}

The following is the main theorem which solves 
the rationality problem of two-dimensional 
quasi-monomial actions under the condition 
that $c_j(\sigma)\in k\setminus\{0\}$ where
$c_j(\sigma)$ is given as in {\rm (iii)} of Definition \ref{defqm}. 
Note that we study only the case 
$G\neq H$ where 
$H=\{ \sigma \in G \ |\ \sigma(\alpha)=\alpha
\ {\rm for} \ {\rm any}\ \alpha \in K \}$ 
because the rationality problem for the case $G=H$ 
was already solved affirmatively by Hajja (see Theorem \ref{thHaj8387}). 
For the norm residue symbol $(a,b)_{n,k}$ of degree $n$ over $k$, 
see Section \ref{ssSB}. 

\begin{theorem}\label{thmain}
Let $k$ be a field with ${\rm char}$ $k \neq 2$, $3$,
$K/k$ be a finite extension
and $G$ be a finite subgroup of
$GL_2(\bZ)$ acting on $K(x,y)$ by quasi-monomial $k$-automorphisms. 
Let $(a,b)_{n,k}$ be the norm residue symbol of degree $n$ over $k$.
Let $\omega$ be a primitive cubic root of unity in $\overline{k}$. 
We assume that $c_j(\sigma)\in k\setminus\{0\}$
for any $\sigma \in G$ and any $1\leq j\leq 2$ 
where
$c_j(\sigma)$ is given as in {\rm (iii)} of Definition \ref{defqm}, 
and 
$G\neq H$ where
$H=\{ \sigma \in G \ |\ \sigma(\alpha)=\alpha
\ {\rm for} \ {\rm any}\ \alpha \in K \}$.\\
{\rm (1)} When $G=C_2^{(1)}=\langle-I\rangle$ and $H=\{1\}$, we may assume that $K=k(\sqrt{a})$ and
$-I : \sqrt{a}\mapsto -\sqrt{a}$, $x\mapsto b/x$, $y\mapsto c/y$
for $a,b,c \in k\setminus\{0\}$.
Then $K(x,y)^G$ is $k$-rational if and only if
$(a,b)_{2,k}=0$ and $(a,c)_{2,k}=0$.\\
{\rm (2)} When $G=C_2^{(2)}=\langle\lambda\rangle$ and $H=\{1\}$, we may assume that $K=k(\sqrt{a})$ and
$\lambda : \sqrt{a}\mapsto -\sqrt{a}$, $x\mapsto x$, $y\mapsto b/y$
for $a,b\in k\setminus\{0\}$.
Then $K(x,y)^G$ is $k$-rational if and only if $(a,b)_{2,k}=0$.\\
{\rm (3)} When $G=C_2^{(3)}=\langle\tau\rangle$ and $H=\{1\}$, we may assume that $K=k(\sqrt{a})$ and
$\tau: \sqrt{a}\mapsto -\sqrt{a}$, $x\mapsto y$, $y\mapsto x$ 
for $a \in k\setminus\{0\}$. 
Then $K(x,y)^G$ is $k$-rational.\\
{\rm (4)} When $G=C_3=\langle\rho^2\rangle$ and $H=\{1\}$, 
we may assume that $\rho^2 : x\mapsto y, y\mapsto c/(xy)$ for $c\in k\setminus\{0\}$. 
We may also assume that 
$K(\omega)=k(\omega,\sqrt[3]{\alpha})$ for some $\alpha\in k(\omega)$ 
and $\rho^2: \sqrt[3]{\alpha}\mapsto \omega\sqrt[3]{\alpha}$. 
Then 
$K(x,y)^G$ is $k$-rational if and only if $(\alpha,c)_{3,k(\omega)}=0$.\\ 
{\rm (5)} When $G=C_4=\langle\sigma\rangle$, 
we may assume that 
$\sigma : x\mapsto y$, $y\mapsto c/x$ for $c\in k\setminus\{0\}$.

{\rm (I)} if $H=\{1\}$, then
we may also assume that $K=k(\alpha,\beta)$ and
$\sigma : \alpha\mapsto \beta$, $\beta\mapsto -\alpha$. 
We have $k(\alpha^2)=K^{\langle\sigma^2\rangle}$ with $[k(\alpha^2):k]=2$. 
Then $K(x,y)^G$ is $k$-rational if and only if $(\alpha^2,c)_{2,k(\alpha^2)}=0$; 

{\rm (II)} if $H=\langle \sigma^2 \rangle$,
then we may also assume that $K=k(\sqrt{a})$ and $\sigma : \sqrt{a}\mapsto -\sqrt{a}$ 
for $a\in k\setminus\{0\}$.
Then $K(x,y)^G$ is $k$-rational if and only if 
$(a,c)_{2,k}=0$ and $(a,-c)_{2,k}=0$.\\
{\rm (6)} When $G=C_6=\langle\rho\rangle$, we may assume that $\rho : x\mapsto xy$, $y\mapsto 1/x$. 
Then $K(x,y)^G$ is $k$-rational.\\
{\rm (7)} When $G=V_4^{(1)}=\langle\lambda,-I\rangle=\langle\lambda,-\lambda\rangle$,

{\rm (I)} if $H=\{1\}$, then we may assume that $K=k(\sqrt{a},\sqrt{b})$ and
$\lambda : \sqrt{a}\mapsto -\sqrt{a}$, $\sqrt{b}\mapsto \sqrt{b}$,
$x\mapsto x$, $y\mapsto d/y$,
$-\lambda : \sqrt{a}\mapsto \sqrt{a}$, $\sqrt{b}\mapsto -\sqrt{b}$,
$x\mapsto c/x$, $y\mapsto y$ for $a,b,c,d\in k\setminus\{0\}$.
Then $K(x,y)^G$ is $k$-rational if and only if
$(a,d)_{2,k}=0$ and $(b,c)_{2,k}=0$;

{\rm (II)} if $H=\langle -I \rangle$, then we may assume that $K=k(\sqrt{a})$ and
$\lambda : \sqrt{a}\mapsto -\sqrt{a}$, $x\mapsto x$, $y\mapsto d/y$,
$-I : \sqrt{a}\mapsto \sqrt{a}$, $x\mapsto c/x$, $y\mapsto d/y$
for $a,c,d\in k\setminus\{0\}$.
Then $K(x,y)^G$ is $k$-rational if and only if
$(a,d)_{2,k(\sqrt{cd})}=0$;

{\rm (III)} if $H=\langle \lambda \rangle$,
then we may assume that $K=k(\sqrt{a})$ and
$\lambda : \sqrt{a}\mapsto \sqrt{a}$, $x\mapsto \ep_1x$, $y\mapsto  d/y$,
$-I : \sqrt{a}\mapsto -\sqrt{a}$, $x\mapsto c/x$, $y\mapsto d/y$
for $a,c,d\in k\setminus\{0\}$, $\ep_1=\pm 1$.

If $\ep_1=1$, then $K(x,y)^G$ is $k$-rational if and only if
$(a,c)_{2,k}=0$.

If $\ep_1=-1$, then $K(x,y)^G$ is $k$-rational if and only if
$(a,-c)_{2,k(\sqrt{ad})}=0$;

{\rm (IV)} if $H=\langle -\lambda \rangle$,
then we may assume that $K=k(\sqrt{a})$ and
$-\lambda : \sqrt{a}\mapsto \sqrt{a}$, $x\mapsto c/x$,
$y\mapsto \ep_2 y$,
$-I : \sqrt{a}\mapsto -\sqrt{a}$, $x\mapsto c/x$,
$y\mapsto d/y$ for $a,c,d\in k\setminus\{0\}$, $\ep_2 =\pm 1$.

If $\ep_2=1$, then $K(x,y)^G$ is $k$-rational if and only if
$(a,d)_{2,k}=0$.

If $\ep_2=-1$, then $K(x,y)^G$ is $k$-rational if and only if
$(a,-d)_{2,k(\sqrt{ac})}=0$.\\
{\rm (8)} When $G=V_4^{(2)}=\langle\tau,-I\rangle$, 
we may assume that $\tau : x\mapsto y$, $y\mapsto x$,
$-I : x\mapsto c/x$, $y\mapsto c/y$ for $c\in k\setminus\{0\}$. 

{\rm (I)} if $H=\{1\}$, then we may also assume that $K=k(\sqrt{a},\sqrt{b})$ and
$\tau : \sqrt{a}\mapsto -\sqrt{a}$, $\sqrt{b}\mapsto \sqrt{b}$, 
$-I : \sqrt{a}\mapsto \sqrt{a}$, $\sqrt{b}\mapsto -\sqrt{b}$ for $a,b\in k\setminus\{0\}$.
Then $K(x,y)^G$ is $k$-rational if and only if $(b,c)_{2,k(\sqrt{a})}=0$;

{\rm (II)} if $H=\langle \tau \rangle$, $\langle -I \rangle$ or $\langle -\tau \rangle$,
then $K(x,y)^G$ is $k$-rational.\\
{\rm (9)} When $G=S_3^{(1)}=\langle\rho^2,\tau\rangle$, 
we may assume that 
$\rho^2: x\mapsto y, y\mapsto c/(xy)$, $\tau: x\mapsto y, y\mapsto x$ 
for $c\in k\setminus\{0\}$. 

{\rm (I)} if $H=\{1\}$, then 
we may also assume that 
$K(\omega)=F(\omega,\sqrt[3]{\alpha})$ for some $\alpha\in F(\omega)$ 
and $\rho^2: \sqrt[3]{\alpha}\mapsto \omega\sqrt[3]{\alpha}$ where 
$F=K^{\langle\rho^2\rangle}$ with $[F:k]=2$ and $F(\omega)=k(\alpha,\omega)$.
Then 
$K(x,y)^G$ is $k$-rational if and only if $(\alpha,c)_{3,k(\alpha,\omega)}=0$; 

{\rm (II)} if $H=\langle \rho^2 \rangle$, then $K(x,y)^G$ is $k$-rational.\\
{\rm (10)} When $G=S_3^{(2)}=\langle\rho^2,-\tau\rangle$, we may assume that 
$\rho^2 :  x\mapsto y$, $y\mapsto 1/(xy)$, 
$-\tau :  x\mapsto 1/y$, $y\mapsto 1/x$.
Then $K(x,y)^G$ is $k$-rational.\\
{\rm (11)} When $G=D_4=\langle\sigma,\tau\rangle$, 
we may assume that $\sigma : x\mapsto y$, $y\mapsto c/x$, 
$\tau : x\mapsto \ep y$, $y\mapsto \ep x$ for $c\in k\setminus\{0\}$, $\ep=\pm 1$. 

{\rm (I)} if $H=\{1\}$, then we may also assume that 
$K=k(\alpha,\beta)$ and
$\sigma : \alpha\mapsto\beta$, $\beta\mapsto -\alpha$, 
$\tau : \alpha\mapsto\beta$, $\beta\mapsto\alpha$. 
We have $k(\alpha^2)=K^{\langle\sigma^2,\sigma\tau\rangle}$ 
with $[k(\alpha^2):k]=2$. 

If $\ep=1$, then
$K(x,y)^G$ is $k$-rational if and only if 
$(\alpha^2,c)_{2,k(\alpha^2)}=0$. 

If $\ep=-1$,  then
$K(x,y)^G$ is $k$-rational if and only if 
$(\alpha^2,-\beta^2c)_{2,k(\alpha^2)}=0$; 

{\rm (II)} if $H=\langle -I \rangle$,
then we may also assume that $K=k(\sqrt{a},\sqrt{b})$ and
$\sigma : \sqrt{a}\mapsto -\sqrt{a}$, $\sqrt{b}\mapsto \sqrt{b}$, 
$\tau : \sqrt{a}\mapsto \sqrt{a}$, $\sqrt{b}\mapsto -\sqrt{b}$ for $a,b\in k\setminus\{0\}$. 
Then $K(x,y)^G$ is $k$-rational if and only if
$(a,\ep c)_{2,k}=0$ and $(a,-\ep bc)_{2,k}=0$;

{\rm (III)} if $H=\langle -I, \tau \rangle$,
then we may also assume that $K=k(\sqrt{a})$ and
$\sigma : \sqrt{a}\mapsto -\sqrt{a}$, 
$\tau : \sqrt{a}\mapsto \sqrt{a}$ for $a\in k\setminus\{0\}$.
Then $K(x,y)^G$ is $k$-rational if and only if $(a,\ep c)_{2,k}=0$;

{\rm (IV)} if $H=\langle -I, \tau\sigma \rangle$ or $\langle\sigma\rangle$, 
then $K(x,y)^G$ is $k$-rational.\\
{\rm (12)} When $G=D_6=\langle\rho,\tau\rangle$, 
we may assume that 
$\rho : x\mapsto  xy$, $y\mapsto 1/x$, 
$\tau : x\mapsto y$, $y\mapsto x$. 
Then $K(x,y)^G$ is $k$-rational.

Moreover, if $K(x,y)^G$ is not $k$-rational, then $k$ is an infinite
field, the Brauer group ${\rm Br}(k)$ is non-trivial, and $K(x,y)^G$
is not $k$-unirational.
\end{theorem}

\begin{remark}\label{r1.7}
(1) In Theorem \ref{thmain}, 
when $H=\{1\}$ and $c_j(\sigma)=1$ for any $\sigma\in G$ and any $1\leq j\leq 2$, 
i.e. the action of $G$ is faithful on $K$ and purely monomial, $K(x,y)^G$ is isomorphic to 
the function field of an algebraic $k$-torus $T$ with ${\rm dim}\, T=2$ 
(see Hoshi, Kang and Kitayama \cite[Section 1]{HKK}, Hoshi \cite[Section 1]{Ho}, 
Hoshi and Yamasaki \cite[Chapter 1]{HY1}). 
In this case, $k(T)\simeq K(x,y)^G$ is $k$-rational by Voskresenskii's theorem (see Theorem \ref{thVos67}).\\
(2) When $H=\{1\}$, because $K(x,y)^G\otimes_k K=K(x,y)$ is $K$-rational, 
one of the referees suggested to the authors that the proof of Theorem \ref{thmain} 
may be modified from the viewpoint of algebraic geometry.  
He/She told that for each case one can consider equivariant compactification of the corresponding 
 (split) algebraic $K$-torus $T\otimes_k K$ and apply equivariant MMP, 
then get a certain  (split) $G$-minimal toric surface over $K$, and 
in Theorem \ref{thmain}, for the cases (1) $C_2^{(1)}$, 
(2) $C_2^{(2)}$, (7) $V_4^{(1)}$ 
it is isomorphic to $\bP^1_K\times \bP^1_K$ with the $G$-invariant Picard number $2$; 
for the cases (3) $C_2^{(3)}$, (5) $C_4$, (8) $V_4^{(2)}$, (11) $D_4$ 
it is isomorphic to 
$\bP^1_K\times \bP^1_K$ with the $G$-invariant Picard number $1$ 
($=2-1$; there is an element in the group that maps one projective line to the other); 
for the cases (4) $C_3$, (10) $S_3^{(2)}$ 
it is isomorphic to $\bP^2_K$; 
and 
for the cases (6) $C_6$, (9) $S_3^{(1)}$, 
(12) $D_6$ it is a $G$-minimal del Pezzo surface of degree $6$ over $K$ 
(see Colliot-Th\'el\`ene, 
Karpenko and Merkurjev \cite[Section 2, Section 4, Proposition 5.3]{CTKM}, 
Xie \cite[Section 3, Section 4]{Xi}, 
see also Iskovskikh \cite{Is4}, 
Mori \cite[Theorem 2.7]{Mo}, 
Manin \cite[Theorem 29.4, Appendix]{Ma}, 
Voskresenskii and Klyachko \cite{VK}, 
Koll\'ar \cite[III.2, III.3]{Ko1}, 
Colliot-Th\'el\`ene \cite[Chapitres I, II, III]{CT}, 
Dolgachev and Iskovskikh \cite{DI}). 

In particular, 
the cases (4) $C_3$, (10) $S_3^{(2)}$ 
correspond to Severi-Brauer surfaces, i.e. $k$-forms of $\bP^2_k$, 
and also to $K_\alpha(I_{G/H})^G$ where 
$I_{G/H}={\rm Ker}(\ep)$ is the kernel of 
the augmentation map $\ep$: 
$0\to I_{G/H}\to \bZ[G/H]\xrightarrow{\ep} \bZ\to 0$ with 
$G=C_3$ or $S_3$, and $[G:H]=3$
(see Section \ref{ssSB} and Hoshi and Yamasaki \cite[Section 1]{HY1}). 
On the other hand,  
the case (9) $S_3^{(1)}$ 
corresponds to del Pezzo surfaces of degree $6$ 
and to $K_\alpha(J_{G/H})^G$ where 
$J_{G/H}=(I_{G/H})^\circ={\rm Hom}_\bZ(I_{G/H},\bZ)$ 
is the dual lattice of $I_{G/H}$,  
and $K(J_{G/H})^G$ (purely quasi-monomial case) 
can be regarded as the function field of 
norm one tori $R_{K/k}^{(1)}(\bG_{m,K})$ 
(see Hoshi and Yamasaki \cite{HY2}). 
Note that the action of $G$ on $K(x_1,\ldots,x_n)$ 
is given via $[a_{i,j}]_{1\leq i,j\leq n}\in GL_n(\bZ)$ 
``vertically'' in this paper (Definition \ref{defqm}) 
as in Kunyavskii \cite{Ku}, \cite{HKK}, \cite{HKiY} 
although it is given ``horizontally'' as in the papers 
Roquette \cite{Ro}, 
Endo and Miyata \cite{EM}, \cite{HY1}, \cite{HY2}, 
(we should switch the case $S_3^{(1)}$ (resp. $J_{G/H}$) 
with $S_3^{(2)}$ (resp. $I_{G/H}$) each other, see also 
\cite[page 174]{HY1}, \cite[Section 2]{Ku}). 
\end{remark}

As an application, we obtain some rationality result 
of $k_\alpha(M)^G$ up to $5$-dimensional cases 
under the action of monomial $k$-automorphisms. 
\begin{corollary}\label{cor1.8}
Let $G$ be one of the finite subgroups 
$C_2^{(3)}=\langle\tau\rangle$, 
$C_6=\langle\rho\rangle$, 
$S_3^{(2)}=\langle\rho^2,-\tau\rangle$, 
$D_6=\langle\rho,\tau\rangle$ of $GL_2(\bZ)$ 
as in Theorem \ref{thmain} $(3)$, $(6)$, $(10)$, $(12)$. 
Let $k$ be a field with {\rm char} $k\neq 2, 3$, 
$M$ be a $G$-lattice with 
$M=M_1\oplus M_2$ as $\bZ[G]$-modules where
$1\leq {\rm rank}_\bZ \, M_1\leq 3$, ${\rm rank}_\bZ \, M_2=2$ and 
$G$ act on $k_\alpha(M)$ by monomial $k$-automorphisms
where 
$M_2$ is a faithful $G$-lattice 
and the action of $G\leq GL_2(\bZ)$ 
on $k_\alpha(M_2)$ is 
given as in {\rm (iii)} of Definition \ref{defqm}.\\ 
{\rm (1)} If ${\rm rank}_\bZ \, M_1=1$ or $2$, 
then $k_\alpha(M)^G$ is $k$-rational;\\
{\rm (2)} If ${\rm rank}_\bZ \, M_1=3$, then 
$k_\alpha(M)^G$ is $k$-rational except for the case 
$G=D_6$ and the action of $G$ on $k_\alpha(M_1)$ is given 
as $G_{3,1,1}=\langle\tau_1,\lambda_1\rangle\simeq V_4$ in \cite{HKiY}. 
\end{corollary}

We organize this paper as follows. 
In Section 2, we prepare some known results about rationality problems. 
We also give a brief review of the Severi-Brauer variety 
with some examples and theorems which will be used in the proof of 
Theorem \ref{thmain}.  
In Section 3, the proof of Theorem \ref{thmain} is given 
by case-by-case analysis. 
The proof of Corollary \ref{cor1.8} is also given. 

\begin{acknowledgments}
We would like to thank Ming-chang Kang 
for giving us useful and valuable comments. 
In particular, he explains Section \ref{ssSB} 
about Severi-Brauer varieties 
with some examples and rationality results to us. 
We thank him for his help and generosity. 

We also thank the referees who gave us many valuable comments. 
In particular, one of them detected incorrect transformations of the variables 
in the proof of Theorem \ref{thmain} which leads us to the correct ones 
and also suggested 
a significant improvement of the proof of Theorem \ref{thmain} 
from the viewpoint of algebraic geometry including 
Remarks \ref{r1.7} (2), \ref{r3.1}, \ref{r3.2}, \ref{r3.3}, \ref{r3.4}. 
\end{acknowledgments}

%
\section{Preliminaries}\label{Sec:Preliminaries}
%
\subsection{Some explicit transcendental bases} \label{Subsec:Bases}

\begin{lemma}[{Hashimoto, Hoshi and Rikuna \cite[page 1176]{HHR}}]\label{l2.1}
Let $k$ be a field with {\rm char} $k\neq 2$ and $-I\in GL_2(\bZ)$ act on
$k(x,y)$ by $k$-automorphism 
\begin{align*}
-I : x\mapsto \frac{a}{x},\ y\mapsto \frac{a}{y}\quad (a\in k\setminus\{0\}).
\end{align*}
Then $k(x,y)^{\langle -I\rangle}=k(s,t)$ where
\begin{align*}
s =\frac{xy+a}{x+y},\quad t =\frac{xy-a}{x-y}.
\end{align*}
\end{lemma}

\begin{lemma}[{Hajja and Kang \cite[Lemma 2.7]{HK2}, Kang \cite[Theorem 2.4]{Ka2}}]\label{l2.2}
Let $k$ be a field and $-I\in GL_2(\bZ)$ act on
$k(x,y)$ by $k$-automorphism
\begin{align*}
-I : x\mapsto \frac{a}{x},\ y\mapsto \frac{b}{y}\quad
(a,b\in k\setminus\{0\})
\end{align*}
where $b=c(x+\frac{a}{x})+d$ with $(c,d)\neq (0,0)$.
Then $k(x,y)^{\langle -I\rangle}=k(u,v)$ where
\begin{align*}
u =\frac{x-\frac{a}{x}}{xy-\frac{ab}{xy}},\quad 
v =\frac{y-\frac{b}{y}}{xy-\frac{ab}{xy}}.
\end{align*}
\end{lemma}

\begin{lemma}[{Hoshi, Kitayama and Yamasaki \cite[Lemma 3.6]{HKiY}}]\label{l2.3}
Let $k$ be a field and $\sigma$ be a $k$-automorphism on $k(x,y)$ defined by
\begin{align*}
\sigma : x\mapsto y\mapsto \frac{b}{xy}\mapsto x\quad (b\in k\setminus\{0\}).
\end{align*}
Then $k(x,y)^{\langle\sigma\rangle}=k(u,v)$ where
\begin{align*}
u&= \frac{y(y^3x^3+bx^3-3byx^2+b^2)}{y^2x^4-y^3x^3+y^4x^2-byx^2-by^2x+b^2},\\
v&= \frac{x(x^3y^3+by^3-3bxy^2+b^2)}{y^2x^4-y^3x^3+y^4x^2-byx^2-by^2x+b^2}.
\end{align*}
\end{lemma}

\begin{lemma}[{Chu, Hoshi, Hu and Kang \cite[page 252, Step 4]{CHHK}}]\label{l2.4}
Let $k$ be a field with {\rm char} $k\neq 3$ and
$\sigma$ be a $k$-automorphism on $k(x,y)$ defined by
\begin{align*}
\sigma : x\mapsto y\mapsto \frac{1}{xy}\mapsto x.
\end{align*}
Assume that $\omega\in k$ where $\omega$ is a primitive 
cubic root of unity in $\overline{k}$.
Then there exist $u,v\in k(x,y)$ such that $k(u,v)=k(x,y)$
and
\[
\sigma: u\mapsto \omega u,\ v\mapsto \omega^{-1} v.
\]
Indeed, we can obtain such $u$ and $v$ as
\begin{align*}
u=\frac{1+\omega^{-1} x+\omega xy}{1+x+xy},\quad
v=\frac{1+\omega x+\omega^{-1} xy}{1+x+xy}.\\
\end{align*}
\end{lemma}

\subsection{A rationality criterion for conic bundles of $\bP^1$} 
%
Let $(a,b)_{2,k}$ be the Hilbert symbol over $k$, 
i.e. the norm residue symbol of degree $2$ over $k$. 
First we recall well-known lemma (see Lam \cite[page 58]{La}).
\begin{lemma}\label{l2.5}
Let $k$ be a field with {\rm char} $k \neq 2$.\\
{\rm (1)} $(a,1-a)_{2,k}=0$ and $(a, -a)_{2,k}=0$ 
for any $a \in k \setminus \{0 \}$.\\
{\rm (2)} For any $a, b \in k \setminus \{0 \}$, the following three statements are equivalent:\\
{\rm (i)} $(a, b)_{2,k}=0$;\\ 
{\rm (ii)} the quadratic form $X_1^2 -aX_2^2 -bX_3^2$ has a non-trivial zero in $k$;\\
{\rm (iii)} the quadratic form $X_1^2 -aX_2^2 -bX_3^2 +abX_4^2$ has a non-trivial zero in $k$.
\end{lemma}
The following is a fundamental result which 
we will use in the proof of Theorem \ref{thmain}. 
\begin{theorem}[{Hajja, Kang and Ohm \cite[Theorem 6.7]{HKO}, Kang \cite[Theorem 4.2]{Ka3}, see also Hoshi, Kang and Kitayama \cite[Lemma 4.1]{HKK} for the last statement}]\label{t2.6}
Let $k$ be a field with {\rm char} $k\neq 2$
and $K=k(\sqrt{a})$ be a quadratic field extension of $k$.
Take ${\rm Gal}(K/k)= \langle \sigma \rangle$ and
extend the action of $\sigma$ to $K(x,y)$ by
\[
\sigma: \sqrt{a} \mapsto -\sqrt{a},\ x \mapsto x,\ y\mapsto
\frac{f(x)}{y}\quad (f(x)\in k[x]).
\]
Then we have $K(x,y)^{\langle\sigma\rangle}=k(z_1,z_2,x)$ 
with the relation $z_1 ^2-az_2 ^2=f(x)$ where 
$z_1=\frac{1}{2}(y+\frac{f(x)}{y})$, 
$z_2=\frac{1}{2\sqrt{a}}(y-\frac{f(x)}{y})$.\\
{\rm (1)} When $f(x)=b$, $K(x,y)^{\langle\sigma\rangle}$ is
$k$-rational if and only if 
$(a,b)_{2,k}=0$.\\
{\rm (2)} When $\deg f(x)=1$,
$K(x,y)^{\langle\sigma\rangle}$ is always $k$-rational.\\
{\rm (3)} When 
$f(x)=b(x^2-c)$ for some $b,c\in k\setminus \{0\}$, then
$K(x,y)^{\langle\sigma\rangle}$ is $k$-rational if and only
if $(a,b)_{2,k}\in {\rm Br}(k(\sqrt{ac})/k)$, 
i.e. $(a,b)_{2,k(\sqrt{ac})}=0$. 

Moreover, if $K(x,y)^{\langle\sigma\rangle}$ is not $k$-rational,
then $k$ is an infinite field, the Brauer group ${\rm Br}(k)$ is
non-trivial, and $K(x,y)^{\langle\sigma\rangle}$ is not
$k$-unirational.
\end{theorem}
Theorem \ref{t2.6} is related to the rationality of conic bundles of 
$\bP^1$ over a non-closed field $k$ investigated by Iskovskikh 
\cite{Is1}, \cite{Is2}, \cite{Is3}, \cite{Is4}, 
see also Hoshi, Kang, Kitayama and Yamasaki \cite[Section 1, Lemma 2.1]{HKKY}, 
Yamasaki \cite{Ya}.

\subsection{Severi-Brauer varieties and rationality}\label{ssSB}

We will give a brief review of the Severi-Brauer variety. 
For a detailed discussion, 
see Serre \cite[page 160]{Se}, Kang \cite{Ka1}, 
Saltman \cite[Chapter 13]{Sa3}, Gille and Szamuely \cite[Chapter 5]{GS} 
and Koll\'{a}r \cite{Ko2}. 

Let $k$ be a field, $\mathrm{Br}(k):=H^2(\Gamma_k,k_{\rm sep}^{\times})$ be its Brauer group
where $k_{\rm sep}$ is the separable closure of $k$ and $\Gamma_k=\mathrm{Gal}(k_{\rm sep}/k)$.
Recall that $\mathrm{Br}(k)=\displaystyle\bigcup_L H^2(\mathrm{Gal}(L/k),L^{\times})$
where $L$ runs over finite Galois extensions of $k$.
If $L/k$ is a Galois extension with $G:=\mathrm{Gal}(L/k)$
and $m$ is any positive integer,
the exact sequence $1\to L^{\times}\to GL_m(L)\to PGL_m(L)\to 1$ gives rise to the injective map
$\delta_m: H^1(G,PGL_m(L))\to H^2(G,L^{\times})$
(see Serre \cite[page 158]{Se} and Roquette \cite[Theorem 1]{Ro});
moreover, we have $H^2(G,L^{\times})=\displaystyle\bigcup_{m\in\bN}\mathrm{Image}(\delta_m)$.

For any $2$-cocycle $\gamma: G\times G\to L^{\times}$,
denote by $[\gamma]$ the cohomology class associated to $\gamma$.
If $[\gamma]\in H^2(G,L^{\times})$ and $[\gamma]=\delta_m([\beta])$
where $m\in\bN$ and $[\beta]\in H^1(G,PGL_m(L))$, 
we may define the Severi-Brauer variety over $k$ corresponding to $[\gamma]$ (and also $[\beta]$)
as in Kang \cite[Definition 2]{Ka1} and Gille and Szamuely \cite[Section 5.2]{GS};
it is denoted by $V_m(\gamma)$ or $V_m(\beta)$.
The function field of $V_m(\gamma)$ is denoted by $F_m(\gamma)$
which is called the {\it $m$-th Brauer field of $[\gamma]$} 
in Roquette \cite[page 412]{Ro}.
It is known that $V_m(\gamma)\otimes_k \mathrm{Spec}(L)$ and $\mathbb{P}^{m-1}\otimes_k \mathrm{Spec}(L)$
are isomorphic as schemes over $\mathrm{Spec}(L)$.
Moreover, the following statements are equivalent 
(see Serre \cite[page 161]{Se}, Roquette \cite[Theorem 2.5]{Ro}, Saltman \cite[Theorem 13.8]{Sa3} and 
Gille and Szamuely \cite[Theorem 5.1.3]{GS}):\\
(i) $V_m(\gamma)$ is isomorphic to $\mathbb{P}^{m-1}$ over $k$;\\
(ii) $V_m(\gamma)$ has a $k$-rational point;\\
(iii) $F_m(\gamma)$ is $k$-rational;\\
(iv) the cohomology class $[\gamma]\in H^2(G,L^{\times})$ is zero.

\begin{example}\label{ex2.8}
We will explain how to construct $F_m(\gamma)$.
Suppose $L/k$ is a Galois extension with $G=\mathrm{Gal}(L/k)$
and $[\gamma]=\delta_m([\beta])$ for some $1$-cocycle $\beta: G\to PGL_m(L)$.
Consider the projection $GL_m(L)\to PGL_m(L)$.
For each $\sigma \in G$, $\beta(\sigma)\in PGL_m(L)$.
Choose a matrix $T_{\sigma}\in GL_m(L)$ such that $T_{\sigma}$ is a preimage of $\beta(\sigma)$.
Let $L(x_1,\ldots, x_m)$ be the rational function field of $m$ variables over $L$.
We define an action of $G$ on $L(x_1,\ldots, x_m)$.
For each $\sigma\in G$, $\sigma$ acts on $L$ by the prescribed Galois extension;
define $\sigma\cdot x_i=x_i$ for $1\leq i\leq m$.
We still use the notation $\sigma$ for the action defined above.
Now $T_{\sigma}$ is naturally a linear map on the vector space $\displaystyle\bigoplus_{1\leq i\leq m}L\cdot x_i$.
Thus $T_{\sigma}$ defines a map $T_{\sigma}: L(x_1,\ldots, x_m) \to L(x_1,\ldots, x_m)$.
Define $u_{\sigma}: L(x_1,\ldots, x_m)\to L(x_1,\ldots, x_m)$ by $u_{\sigma}=T_{\sigma}\cdot \sigma$.
Then each $u_{\sigma}$ is a $k$-algebra automorphism of $L(x_1,\ldots, x_m)$
and $u_{\sigma}$ preserves the degrees of polynomials in $L[x_1,\ldots, x_m]$.
It follows that, for any homogeneous polynomials
$f, g\in L[x_1,\ldots, x_m]\setminus \{0\}$ of the same degree,
it is easy to show that $(u_{\sigma}\cdot u_{\tau})(f/g)=u_{\sigma\tau}(f/g)$
where $\sigma, \tau\in G$.
Thus $\widehat{G}:=\{ u_{\sigma}\ |\  \sigma\in G\}$ is an automorphism subgroup of $L(x_i/x_m\ |\ 1\leq i\leq m-1)$.
The field $F_m(\gamma)$ is defined as $F_m(\gamma):=L(x_i/x_m\ |\ 1\leq i\leq m-1)^{\widehat{G}}$.
For details, see Roquette \cite[Section 5]{Ro} and Kang \cite[page 232]{Ka1}.
\end{example}

\begin{example}\label{ex2.9}
Consider the special case of a cyclic extension $L/k$ where $G=\mathrm{Gal}(L/k)=\langle\sigma\rangle\simeq C_n$.
For any $b\in k^{\times}$, define a matrix $T\in GL_n(L)$ by
\[
T=\left(\begin{array}{@{}ccc;{3pt/2pt}c@{}}
0&\cdots & 0 & b\\\hdashline[3pt/2pt]
1&  & &0\\
& \ddots & &\vdots\\
&  & 1& 0
\end{array}\right)\in GL_n(L).
\]
Regard $T$ as an element in $PGL_n(L)$.
Define a $1$-cocycle $\beta$
\[
\begin{array}{cccl}
\beta: & G & \longrightarrow & PGL_n(L) \\
       & \sigma^i & \mapsto & T^i
\end{array}
\]
where $0\leq i\leq n-1$.
Let $\gamma=\delta_n(\beta)$ be the corresponding $2$-cocycle.
It follows that, for $0\leq i, j\leq n-1$,
\[
\begin{array}{cccl}
\gamma: & G\times G & \longrightarrow & L^{\times} \\
 & (\sigma^i,\sigma^j) & \mapsto &
\begin{cases}
1 &{\rm if}\quad 0\leq i+j\leq n-1, \\
b &{\rm if}\quad i+j\geq n.
\end{cases}
\end{array}
\]
Thus the function field $F_n(\gamma)$ of $V_n(\beta)$ corresponds to the cyclic algebra
$(b,L/k,\sigma)$ which is defined as
\[
(b,L/k,\sigma)=\bigoplus_{0\leq i\leq n-1}L\cdot v^i,
\ \ v^n=b, \ \ v\cdot\alpha=\sigma(\alpha)\cdot v
\]
for any $\alpha\in L$ (see Draxl \cite[page 49, Definition 4]{Dr}).

It is routine to define $F_n(\gamma$) explicitly.
In $L(v_i\ |\ 1\leq i\leq n-1)$, define $y_i=v_{i+1}/v_i$
where $v_i=x_i/x_n$ for $1\leq i\leq n-1$ and $v_n=1$.
From the definition of $T$, we find that
\[
u_{\sigma}: y_1\mapsto y_2\mapsto \cdots \mapsto y_{n-1}\mapsto 
\frac{b}{y_1\cdots y_{n-1}}\mapsto y_1.
\]
Thus $F_n(\gamma)=L(y_1,\ldots, y_{n-1})^{\langle u_{\sigma}\rangle}$.

In particular, if $\mathrm{gcd}\{\mathrm{char}\ k, n\}=1$ and $k$ contains a primitive $n$-th root of unity 
$\zeta_n\in\overline{k}$, then the cyclic extension $L$ of $k$ may be written as
$L=k(\sqrt[n]{a})$ for some $a\in k\setminus k^n$ with 
$\mathrm{Gal}(k(\sqrt[n]{a})/k)=\langle\sigma\rangle\simeq C_n$, \ $\sigma\cdot \sqrt[n]{a}=\zeta_n \sqrt[n]{a}$.
The cyclic algebra $(b,L/k,\sigma)$ becomes the power norm residue algebra $(a,b;n,k,\zeta_n)$
\[
(a,b;n,k,\zeta_n):=\bigoplus_{0\leq i,j\leq n-1} k\cdot u^iv^j, \ u^n=a, \ v^n=b, \ uv=\zeta_n vu.
\]
The class of $(a,b;n,k,\zeta_n)$ in $\mathrm{Br}(k)$ is denoted by $[a,b;n,k,\zeta_n]$
(see Draxl \cite[page 78]{Dr} for details).
Thus the fixed field $k(\sqrt[n]{a})(y_1,\ldots, y_{n-1})^{\langle u_{\sigma}\rangle}$
is $k$-rational if and only if $[a,b;n,k,\zeta_n]=0$ in $\mathrm{Br}(k)$. 
For simplicity, we just write $(a,b)_{n,k}=[a,b;n,k,\zeta_n]$ 
and call it {\it the norm residue symbol of degree $n$ over $k$}.
When $n=2$, $(a,b)_{2,k}=[a,b;n,k,-1]$ is nothing but the Hilbert symbol.
\end{example}

In summary we record the following theorem: 

\begin{theorem}\label{t2.10}
Let $n$ be a positive integer, 
$k$ be a field with $\mathrm{gcd}\{\mathrm{char}\ k, n\}$ $=$ $1$ 
and $\zeta_n \in k$ where $\zeta_n$ is a primitive $n$-th root of unity in $\overline{k}$. 
Let $(a,b)_{n,k}$ be the norm residue symbol of degree $n$ over $k$. 
Suppose that $k(\sqrt[n]{a})$ is a cyclic extension of $k$ of degree $n$ 
with $a \in k \setminus \{ 0 \}$. 
Take ${\rm Gal}(k(\sqrt[n]{a})/k)=\langle\sigma\rangle$ and
extend the action of $\sigma$ to $k(\sqrt[n]{a})(x_1,\ldots,x_{n-1})$ by
\[
\sigma: x_1\mapsto \cdots\mapsto x_{n-1}\mapsto
\frac{b}{x_1\cdots x_{n-1}}\mapsto x_1
\]
where $b\in k\setminus\{0\}$. 
Then the following statements are equivalent:\\
{\rm (i)} $k(\sqrt[n]{a})(x_1, \ldots, x_{n-1})^{\langle \sigma \rangle}$ is $k$-rational;\\
{\rm (ii)} $k(\sqrt[n]{a})(x_1, \ldots, x_{n-1})^{\langle \sigma \rangle}$ is stably $k$-rational;\\
{\rm (iii)} $k(\sqrt[n]{a})(x_1, \ldots, x_{n-1})^{\langle \sigma \rangle}$ is $k$-unirational;\\
{\rm (iv)} $(a,b)_{n,k}=0$.
\end{theorem} 

Theorem \ref{t2.11} and Theorem \ref{t2.12} 
are applications of the foregoing discussion 
which we will use in the proof of Theorem \ref{thmain}. 

\begin{theorem}\label{t2.11} 
Let $k$ be a field with {\rm char} $k\neq 2$ and
$K=k(\sqrt{a})$ be a quadratic field extension of $k$.
Let $(a,b)_{2,k}$ be the norm residue symbol of degree $2$ over $k$. 
Take ${\rm Gal}(K/k)=\langle\sigma\rangle$ and extend the action of
$\sigma$ to $K(x_1,\ldots,x_n)$ by
\begin{align*}
\sigma: 
\sqrt{a}\mapsto -\sqrt{a},\
x_i\mapsto \frac{b_i}{x_i}\quad (1\leq i\leq n)
\end{align*}
where $b_i\in k\setminus\{0\}$. 
Then we
have $K(x_1,\ldots,x_n)^{\langle\sigma\rangle}=k(s_i,t_i:1\leq i\leq n)$
where
\begin{align}
s_i^2-at_i^2=b_i\quad (1\leq i\leq n)\label{eq00}
\end{align}
and the following statements are equivalent:\\
{\rm (i)} $K(x_1,\ldots,x_n)^{\langle\sigma\rangle}$ is $k$-rational;\\
{\rm (ii)} $K(x_1,\ldots,x_n)^{\langle\sigma\rangle}$ is stably
$k$-rational;\\
{\rm (iii)} $K(x_1,\ldots,x_n)^{\langle\sigma\rangle}$ is $k$-unirational;\\
{\rm (iv)} the variety defined by $(\ref{eq00})$
has a $k$-rational point;\\
{\rm (v)}
$(a,b_i)_{2,k}=0$ for any $1\leq i\leq n$.
\end{theorem}
\begin{proof}
Note that $L_i=k(s_i, t_i)$ where 
$s_i^2-at_i^2=b_i$ 
is the function field of 
the Severi-Brauer variety associated 
to the quaternion algebra $(a,b_i)_{2,k}$.

Define 
$L_0:=K(x_1,\ldots,x_n)^{\langle\sigma\rangle}=k(s_i,t_i:1\leq i\leq n)$. 
If $L_0$ is $k$-unirational, so is its subfield $L_i$. 
Thus the Hilbert symbol $(a,b_i)_{2,k}=0$ by 
Theorem \ref{t2.10} (see also Serre \cite[page 161]{Se}).

On the other hand, if $(a,b_i)_{2,k}= 0$, then the field $L_i$ is $k$-rational 
by Theorem \ref{t2.10} 
(see also Lemma \ref{l2.5}, Theorem \ref{t2.6} and Serre \cite[page 161]{Se}). 
Since the field $L_0$ is the compositum of the fields 
$L_i$ for $1 \le i \le n$, we find that $L_0$ is $k$-rational also.
\end{proof}
\begin{theorem}\label{t2.12}
Let $k$ be a field with {\rm char} $k\neq 2$,
$K=k(\sqrt{a_1}, \ldots, \sqrt{a_n})$ where
$a_i\in K\setminus\{0\}$ $(1\leq i\leq n)$. 
Let $(a,b)_{2,k}$ be the norm residue symbol of degree $2$ over $k$. 
Let $\sigma_i$ $(1\leq i\leq n)$ be a $k$-automorphism of
$K(x_1, \ldots, x_n)$ defined by
\[
\sigma_i : \sqrt{a_j}\mapsto
\left\{ \begin{array}{ll}-\sqrt{a_j} & {\rm if}\quad i=j \\
\sqrt{a_j} & {\rm if}\quad i\neq j \end{array}\right.,\ 
x_j\mapsto
\left\{ \begin{array}{ll}\displaystyle{\frac{c_j}{x_j}} & {\rm if}\quad i=j \\
x_j & {\rm if}\quad i\neq j \end{array} \right.\quad(1\leq i, j \leq n)
\]
where $c_j\in k\setminus\{0\}$. 
Then we have $K(x_1,\ldots, x_n)^{\langle\sigma_1,\ldots,\sigma_n\rangle}
=k(s_i,t_i:1\leq i\leq n)$ where
\begin{align}
s_i^2-a_it_i^2=c_i\quad (1\leq i\leq n)\label{eq01}
\end{align}
and the following statements are equivalent:\\
{\rm (i)} $K(x_1,\ldots, x_n)^{\langle\sigma_1,\ldots,\sigma_n\rangle}$
is $k$-rational;\\
{\rm (ii)} $K(x_1,\ldots, x_n)^{\langle\sigma_1,\ldots,\sigma_n\rangle}$
is stably $k$-rational;\\
{\rm (iii)} $K(x_1,\ldots, x_n)^{\langle\sigma_1,\ldots,\sigma_n\rangle}$
is $k$-unirational;\\
{\rm (iv)} the variety defined by $(\ref{eq01})$
has a $k$-rational point;\\
{\rm (v)} $(a_i,c_i)_{2,k}=0$ for and $1\leq i\leq n$.
\end{theorem}
\begin{proof}
Note that $L_i=k(\sqrt{a_i})(x_i)^{\langle\sigma_i\rangle}
=k(s_i, t_i)$ where $s_i^2-a_it_i^2=c_i$ 
is the function field of  
the Severi-Brauer variety associated 
to the quaternion algebra $(a_i,c_i)_{2,k}$ 
and $K(x_1,\ldots, x_n)^{\langle\sigma_1,\ldots,\sigma_n\rangle}$ 
is the compositum of the fields $L_i$ for $1 \le i \le n$. 
If $K(x_1,\ldots, x_n)^{\langle\sigma_1,\ldots,\sigma_n\rangle}$ is 
$k$-unirational, so is $L_i$ for any $1 \le i \le n$. 
Thus $(a_i,c_i)_{2,k}=0$ by Theorem \ref{t2.10} 
(see also Serre \cite[page 161]{Se}). 

On the other hand, if $(a_i,c_i)_{2,k}= 0$, 
then the field $L_i$ is $k$-rational 
by Theorem \ref{t2.10} (see also Lemma \ref{l2.5}, Theorem \ref{t2.6} 
and Serre \cite[page 161]{Se}). 
Since the field 
$K(x_1,\ldots, x_n)^{\langle\sigma_1,\ldots,\sigma_n\rangle}$ 
is the compositum of the fields $L_i$, 
we find that it is $k$-rational. 
\end{proof}

%
\section{Proof of Theorem \ref{thmain}}  \label{Sec:Proof}
%

In this section, we will prove Theorem \ref{thmain} 
by case-by-case analysis. 
Recall that $H=\{ \sigma \in G \ |\ \sigma(\alpha)=\alpha
\ {\rm for} \ {\rm any}\ \alpha \in K \}$ 
and we assume that $G\neq H$ (cf. Theorem \ref{thHaj8387}). 

\subsection{The case of $G=C_2^{(1)}=\langle -I \rangle$}\label{ss31}

We have $H=\{1\}$. 
In this case, $K/k$ is a quadratic extension and then $K=k(\sqrt{a})$ for some $a\in k$.
The group $G=\langle-I\rangle$ acts on $k(\sqrt{a})(x,y)$ by
\[
-I : \sqrt{a}\mapsto -\sqrt{a}, \ x\mapsto \frac{b}{x}, \
y\mapsto \frac{c}{y}
\ \ \ \ (b,c \in k\setminus\{0\}).
\]
It follows from Theorem \ref{t2.11} that
$K(x,y)^G$ is $k$-rational if and only if 
$K(x,y)^G$ is $k$-unirational if and only if 
$(a,b)_{2,k}=0$ and $(a,c)_{2,k}=0$.

\subsection{The case of $G=C_2^{(2)}=\langle \lambda \rangle$}\label{ss32}

We have $H=\{1\}$ and $K=k(\sqrt{a})$ for some $a\in k$.
The action of $\lambda$ on $K(x,y)$ is given by
\[
\lambda : \sqrt{a}\mapsto -\sqrt{a}, \ x\mapsto \ep x, \
y\mapsto \frac{b}{y}\quad (\ep=\pm 1, \ b\in k\setminus\{0\}).
\]
By replacing $\sqrt{a}x$ by $x$ if $\ep=-1$,
we may assume that $\ep=1$.
Then it follows from Theorem \ref{t2.6} (Lemma \ref{l2.5}) that
$K(x,y)^G=k(\sqrt{a},x,y)^{\langle \lambda \rangle}$ is $k$-rational 
if and only if $K(x,y)^G$ is $k$-unirational if and only if $(a,b)_{2,k}=0$.

\subsection{The case of $G=C_2^{(3)}=\langle \tau \rangle$}\label{ss33}

We have $H=\{1\}$ and $K=k(\sqrt{a})$ for some $a\in k$.
The action of $\tau$ on $K(x,y)$ is given by
\[
\tau : \sqrt{a}\mapsto -\sqrt{a}, \ x\mapsto by, \ y\mapsto \frac{x}{b}
\quad (b\in k\setminus\{0\}).
\]
By replacing $by$ by $y$,
we may assume that $b=1$.
Hence the problem is reduced to the purely quasi-monomial case,
which has been settled in Theorem \ref{thHKY14}.
Hence $K(x,y)^G$ is $k$-rational.

\subsection{The case of $G=C_3=\langle \rho^2 \rangle$}\label{ss34}

We have $H=\{1\}$. 
The action of $\rho^2$ on $K(x,y)$ is given by
\[
\rho^2 : \ x\mapsto by, \ y\mapsto \frac{c}{xy}
\quad (b,c\in k\setminus\{0\}).
\]
By replacing $by$ by $y$ and $b^2c$ by $c$, we may assume that $b=1$.\\

\noindent 
\textbf{Case 1-1:} $H=\{1\}$ and $\omega\in k$. 
We assume that $\omega\in k$ where 
$\omega$ is a primitive cubic root of unity in $\overline{k}$. 
In this case, $[K:k]=3$ and 
$K=k(\sqrt[3]{a})$ for some $a\in k$. 
It follows from Theorem \ref{t2.10} that 
$K(x,y)^G$ is $k$-rational if and only if 
$K(x,y)^G$ is $k$-unirational if and only if $(a,c)_{3,k}=0$.\\

\noindent 
\textbf{Case 1-2:} $H=\{1\}$ and $\omega\not\in k$, i.e. $\omega\not\in K$. 
We take $L=K(\omega)=k(\omega,\sqrt[3]{\alpha})$ 
for some $\alpha\in k(\omega)$
and ${\rm Gal}(L/K)=\langle\varphi_\omega\rangle$ where 
$\varphi_\omega : \omega\mapsto\omega^{-1}$. 
We extend the actions of $\rho^2$, 
$\varphi_\omega$ to $L(x,y)$ in the natural way. 
Then ${\rm Gal}(L/k)=\langle\rho^2,\varphi_\omega\rangle\simeq C_3\times C_2\simeq C_6$ 
and we find that $L(x,y)^{\langle\rho^2,\varphi_\omega\rangle}
=L^{\langle\varphi_\omega\rangle}(x,y)^G=K(x,y)^G$.\\

(i) Suppose that $(\alpha,c)_{3,k(\omega)}=0$ where $\alpha\in k(\omega)$ 
which satisfies $L=K(\omega)=k(\omega,\sqrt[3]{\alpha})$. 
Then there exists 
$\gamma_0\in k(\omega,\sqrt[3]{\alpha})$ such that 
$c=
N_{k(\omega,\sqrt[3]{\alpha})/k(\omega)}(\gamma_0)=\gamma_0\gamma_1\gamma_2$ 
where $\gamma_1=\rho^2(\gamma_0)$, $\gamma_2=\rho^2(\gamma_1)$. 
Define 
\begin{align*}
X:=\frac{x}{\gamma_0},\quad 
Y:=\frac{y}{\gamma_1}.
\end{align*}
Then $L(x,y)=L(X,Y)$ and the actions of $\rho^2$ and $\varphi_\omega$ on 
$L(X,Y)=k(\omega,\sqrt[3]{\alpha})(X,Y)$ are given by
\begin{align*}
\rho^2 &: \sqrt[3]{\alpha}\mapsto \omega\,\sqrt[3]{\alpha},
\ \sqrt[3]{\beta}\mapsto\omega^{-1}\sqrt[3]{\beta},\ \omega\mapsto\omega,
\ X\mapsto Y, \ Y\mapsto\frac{1}{XY},\\
\varphi_\omega &: \sqrt[3]{\alpha}\mapsto\sqrt[3]{\beta},
\ \sqrt[3]{\beta}\mapsto\sqrt[3]{\alpha},\ \omega\mapsto\omega^{-1},
\ X\mapsto X, \ Y\mapsto Y
\end{align*}
for some $\beta\in k(\omega)$ which satisfies  
$K(\omega)=k(\omega,\sqrt[3]{\alpha})=k(\omega,\sqrt[3]{\beta})$. 
Define
\begin{align*}
u:=\frac{1+\omega^{-1}X+\omega XY}{1+X+XY},\quad
v:=\frac{1+\omega X+\omega^{-1}XY}{1+X+XY}.
\end{align*}
Then it follows from Lemma \ref{l2.4} that
$L(X,Y)=L(u,v)$ and
\begin{align*}
\rho^2 &: \sqrt[3]{\alpha}\mapsto \omega\,\sqrt[3]{\alpha},
\ \sqrt[3]{\beta}\mapsto\omega^{-1}\sqrt[3]{\beta},\ \omega\mapsto\omega,
\ u\mapsto \omega u,\ v\mapsto \omega^{-1}v,\\
\varphi_\omega &: \sqrt[3]{\alpha}\mapsto\sqrt[3]{\beta},
\ \sqrt[3]{\beta}\mapsto\sqrt[3]{\alpha},\ \omega\mapsto\omega^{-1},
\ u\mapsto v, \ v\mapsto u.
\end{align*}
Define
\begin{align*}
s:=\frac{u}{\sqrt[3]{\alpha}},\quad 
t:=\frac{v}{\sqrt[3]{\beta}}. 
\end{align*}
Then we have $L(u,v)=L(s,t)$ and
\begin{align*}
\rho^2 &: \sqrt[3]{\alpha}\mapsto \omega\,\sqrt[3]{\alpha},
\ \sqrt[3]{\beta}\mapsto\omega^{-1}\sqrt[3]{\beta},\ \omega\mapsto\omega,
\ s\mapsto s,\ t\mapsto t,\\
\varphi_\omega &: \sqrt[3]{\alpha}\mapsto\sqrt[3]{\beta},
\ \sqrt[3]{\beta}\mapsto\sqrt[3]{\alpha},\ \omega\mapsto\omega^{-1},
\ s\mapsto t, \ t\mapsto s.
\end{align*}
Hence $K(x,y)^G=L(s,t)^{\langle\rho^2,\varphi_\omega\rangle}
=(L^{\langle\rho^2\rangle}(s,t))^{\langle\varphi_\omega\rangle}=k(\omega)(s,t)^{\langle\varphi_\omega\rangle}
=k(s+t,(\omega-\omega^{-1})(s-t))$ is $k$-rational.\\ 

{\rm (ii)} Suppose that $(\alpha,c)_{3,k(\omega)}\neq 0$ where $\alpha\in k(\omega)$ 
which satisfies that $K(\omega)=k(\omega,\sqrt[3]{\alpha})$. 
If we take a base field $k(\omega)$ instead of $k$, then 
by Case 1-1, 
$K(\omega)(x,y)^G$ is not $k(\omega)$-unirational. 
Hence $K(x,y)^G$ is not $k$-unirational. 

\begin{remark}\label{r3.1}
One of the referees pointed out that 
for the case where $G=C_3$, 
it corresponds to Severi-Brauer surfaces 
as in Example \ref{ex2.8} and Example \ref{ex2.9} 
and it is well known that 
if $F$ is a quadratic extension of $k$, then 
a Severi-Brauer surface $S$ over $k$ is trivial, 
i.e. isomorphic to $\bP^2_k$, 
if and only if $S_F=S\otimes_k F$ over $F$ is trivial 
(see also Roquette \cite{Ro}). 
Thus $K(x,y)^G$ is $k$-rational if and only if 
$K(\omega)(x,y)^G$ is $k(\omega)$-rational if and only if 
$(\alpha,c)_{3,k(\omega)}=0$ where $\alpha\in k(\omega)$ 
is an element such that $\sqrt[3]{\alpha}\in K(\omega)$. 
\end{remark}

\subsection{The case of $G=C_4=\langle \sigma \rangle$}\label{ss35}

The action of $\sigma$ on $K(x,y)$ is given by
\[
x\mapsto by, \ y\mapsto \frac{c}{x}\quad (b,c\in k\setminus\{0\}).
\]
By replacing $by$ by $y$ and $bc$ by $c$, we may assume that $b=1$. 

We will treat the problem
for each of the proper normal subgroups
$H=\{1\}$, $\langle \sigma^2 \rangle=\langle -I\rangle$ of $G$.\\

\noindent
\textbf{Case 1:} $H=\{1\}$.
In this case, $[K:k]=4$ and we can take suitable $\alpha, \beta \in K$ 
which satisfy $K=k(\alpha,\beta)$. 
Then the action of $\sigma$ on $K(x,y)$ is given by
\[
\sigma : \alpha\mapsto \beta,\ \beta\mapsto -\alpha,
\ x\mapsto y, \ y\mapsto \frac{c}{x}\quad (c\in k\setminus\{0\}).
\]
We consider the field 
$F=K^{\langle\sigma^2\rangle}=k(\alpha^2)$ with $[F:k]=2$. 
%
We find that $K(x,y)^G\otimes_k F=(K(x,y)^G)F=K(x,y)^{\langle\sigma^2\rangle}$ 
because $F$ $\cap$ $K(x,y)^G=k$ and $[K(x,y)^{\langle\sigma^2\rangle}:K(x,y)^G]=2$.\\ 

(i) Suppose $(\alpha^2,c)_{2,k(\alpha^2)}=0$. 
Then there exists 
$\gamma_0\in K=k(\alpha,\beta)=k(\alpha)$ such that 
$c=
N_{k(\alpha)/k(\alpha^2)}(\gamma_0)=\gamma_0\gamma_1$ 
where $\gamma_1=\sigma^2(\gamma_0)=-I(\gamma_0)$. 
Define 
\begin{align*}
X:=\frac{x}{\gamma_0},\quad 
Y:=\frac{y}{\sigma(\gamma_0)}.
\end{align*}
Then $K(x,y)=K(X,Y)$ and the actions of 
$\sigma$ and $\sigma^2$ on $K(X,Y)$ are given by 
\begin{align*}
\sigma : \alpha\mapsto \beta,\ \beta\mapsto -\alpha,
\ X\mapsto Y, \ Y\mapsto \frac{1}{X},\\ 
\sigma^2 : \alpha\mapsto -\alpha,\ \beta\mapsto -\beta,
\ X\mapsto \frac{1}{X}, \ Y\mapsto \frac{1}{Y}.
\end{align*}
Hence, by Voskresenskii's theorem (Theorem \ref{thVos67}),  
$K(x,y)^G=K(X,Y)^{\langle\sigma\rangle}$ is $k$-rational. 

Indeed, we define
\[
u:=\alpha\,\frac{X+1}{X-1},\quad 
v:=\beta\,\frac{Y+1}{Y-1}. 
\]
Then the actions of $\sigma$ and $\sigma^2$ 
on $K(x,y)=k(\alpha,\beta)(u,v)$ are given by 
\begin{align*}
\sigma &: \alpha\mapsto \beta,\ \beta\mapsto -\alpha, 
\ u\mapsto v, \ v\mapsto u, \\
\sigma^2 &: \alpha\mapsto -\alpha,\ \beta\mapsto -\beta,
\ u\mapsto u,\ v\mapsto v.
\end{align*}
We have $K(x,y)^{\langle\sigma^2\rangle}=k(\alpha,\beta)(u,v)^{\langle\sigma^2\rangle}=k(A,B)(u,v)$ 
where $A=\alpha\beta$, $B=\alpha/\beta$. 
The action of $\sigma$ on $k(A,B)(u,v)$ is given by 
\begin{align*}
\sigma: A\mapsto -A, \ B\mapsto -\frac{1}{B}, \ u\mapsto v, \ v\mapsto u.
\end{align*}
Define 
\begin{align*}
a&:=A\left(B+\frac{1}{B}\right)=\alpha^2+\beta^2\in k,\ 
b:=B-\frac{1}{B}=\frac{\alpha^2-\beta^2}{\alpha\beta}\in k,\\ 
U&:=u+v, \ V:=A(u-v). 
\end{align*}
Then $a, b\in K^{\langle\sigma\rangle}=k$ and 
$K(x,y)^G=k(A,B)(u,v)^{\langle\sigma\rangle}=k(a,b)(U,V)=k(U,V)$. 
Hence $K(x,y)^G$ is $k$-rational.\\ 

(ii) Suppose $(\alpha^2,c)_{2,k(\alpha^2)}\neq 0$. 
We have $K(x,y)^G\otimes_k F=K(x,y)^{\langle\sigma^2\rangle}$ 
and the action of $\sigma^2$ on $K(x,y)=k(\alpha,\beta)(x,y)=k(\alpha)(x,y)$ is given by 
\begin{align*}
\sigma^2 &: \alpha\mapsto -\alpha, 
\ x\mapsto \frac{c}{x},\ y\mapsto \frac{c}{y}.
\end{align*}
By Section \ref{ss31}, $K(x,y)^G\otimes_k F$ is not $F$-unirational 
where $F=K^{\langle\sigma^2\rangle}=k(\alpha^2)$ with $[F:k]=2$. 
Hence $K(x,y)^G$ is not $k$-unirational.

\begin{remark}\label{r3.2}
One of the referees told that 
the case $G=C_4$ with $H=\{1\}$ corresponds to 
a del Pezzo surface $S$ of degree $8$ and it is well-known that 
$S$ is $k$-rational if and only if $S_F=S\otimes_k F$ is $F$-rational 
where $F=K^{\langle\sigma^2\rangle}$ with $[F:k]=2$. 
It is also interesting to consider a more geometric proof 
and compare it with recent related papers 
Shramov and Vologodsky \cite[Section 7]{SV} and Trepalin \cite{Tr2}. 
\end{remark}

\noindent
\textbf{Case 2:} $H=\langle \sigma^2 \rangle=\langle -I\rangle$.
In this case, $[K:k]=2$ and $K=k(\sqrt{a})$ for some $a\in k$.
The action of $\sigma$ on $K(x,y)$ is given by
\[
\sigma : \sqrt{a}\mapsto -\sqrt{a}, \ x\mapsto y, \ y\mapsto
\frac{c}{x}\quad (c\in k\setminus\{0\}).
\]
By Lemma \ref{l2.1}, we have $k(\sqrt{a},x,y)^{\langle \sigma^2 \rangle}=k(\sqrt{a},u,v)$ where
\[
u=\frac{xy+c}{x+y},\quad v=\frac{xy-c}{x-y}.
\]
The action of $\sigma$ on $k(\sqrt{a},u,v)$ is given by
\[
\sigma: \ \sqrt{a}\mapsto -\sqrt{a}, \ u\mapsto \frac{c}{u}, \ v\mapsto -\frac{c}{v}.
\]
Hence it follows from Theorem \ref{t2.11} that 
$k(\sqrt{a},x,y)^G$ is $k$-rational if and only if 
$k(\sqrt{a},x,y)^G$ is $k$-unirational if and only if 
$(a,c)_{2,k}=0$ and $(a,-c)_{2,k}=0$.

\subsection{The case of $G=C_6=\langle \rho \rangle$}\label{ss36}

We have $K=k(\alpha)$ for some $\alpha\in K$.
The action of $\rho$ on $K(x,y)$ is given by
\[
\rho : \alpha \mapsto \alpha^\prime, \ x\mapsto bxy, \ y\mapsto \frac{c}{x}
\quad (b,c\in k\setminus\{0\}).
\]
By replacing $x/bc$ by $x$ and $by$ by $y$,
we may assume that $b=c=1$.
Therefore the problem is reduced to the purely quasi-monomial case.
It follows from Theorem \ref{thHKY14}
that $K(x,y)^G$ is $k$-rational for each of the proper normal subgroups
$H=\{1\}$, $\langle\rho^3\rangle=\langle -I\rangle$, $\langle\rho^2\rangle$ of $G$.

\subsection{The case of $G=V_4^{(1)}=\langle \lambda, -I \rangle=\langle\lambda,-\lambda\rangle$}\label{ss37}

By the equality 
$\lambda(-I)=(-I)\lambda$,
we see that the action of $G$ on $K(x,y)$ is given by
\[
\lambda : x\mapsto \ep_1x, \ y\mapsto \frac{\ep_2d}{y},\quad 
-I : x\mapsto \frac{c}{x}, \ y\mapsto \frac{d}{y}\quad
(c,d\in k\setminus\{0\},\ \ep_1, \ep_2=\pm 1).
\]

We will treat the problem
for each of the proper normal subgroups
$H=\{1\}$, $\langle -I \rangle$, $\langle \lambda \rangle$, 
$\langle -\lambda \rangle$ of $G$.\\

\noindent
\textbf{Case 1:} $H=\{1\}$.
In this case, $[K:k]=4$ and $K=k(\sqrt{a},\sqrt{b})$ for some $a,b\in k$.
Then the actions of $\lambda$ and $-\lambda$ on $k(\sqrt{a},\sqrt{b})(x,y)$ are given by 
\begin{align*}
\lambda &: \sqrt{a}\mapsto -\sqrt{a}, \ \sqrt{b}\mapsto \sqrt{b}, \
x\mapsto \ep_1x, \ y\mapsto \frac{\ep_2d}{y}, \\
-\lambda &: \sqrt{a}\mapsto \sqrt{a}, \ \sqrt{b}\mapsto -\sqrt{b}, \
x\mapsto \frac{\ep_1 c}{x}, \ y\mapsto \ep_2 y.
\end{align*}
By replacing $\sqrt{a}x$ by $x$ and $-ac$ by $c$ if $\ep_1=-1$
and $\sqrt{b}y$ by $y$ and $-bd$ by $d$ if $\ep_2=-1$,
we may assume that $\ep_1=\ep_2=1$.
Then we have
\begin{align*}
\lambda &: \sqrt{a}\mapsto -\sqrt{a}, \ \sqrt{b}\mapsto \sqrt{b}, \ x\mapsto x, \ y\mapsto \dfrac{d}{y},\\
-\lambda &: \sqrt{a}\mapsto \sqrt{a}, \ \sqrt{b}\mapsto -\sqrt{b}, \ x\mapsto \dfrac{c}{x}, \ y\mapsto y.
\end{align*}
It follows from Theorem \ref{t2.12} that
$K(x,y)^G=K(x,y)^{\langle\lambda,-\lambda\rangle}$ 
is $k$-rational if and only if 
$K(x,y)^G$ is $k$-unirational if and only if 
$(a,d)_{2,k}=0$ and $(b,c)_{2,k}=0$.\\

\noindent
\textbf{Case 2:} $H=\langle -I \rangle$.
In this case, $[K:k]=2$ and $K=k(\sqrt{a})$ for some $a\in k$.
Then the actions of $\lambda$ and $-I$ on $k(\sqrt{a})(x,y)$ are given by 
\begin{align*}
\lambda : \sqrt{a}\mapsto -\sqrt{a}, \ x\mapsto \ep_1x, \ y\mapsto \frac{\ep_2d}{y},\quad 
-I : \sqrt{a}\mapsto \sqrt{a}, \ x\mapsto \frac{c}{x}, \ y\mapsto \frac{d}{y}.
\end{align*}
By replacing $\sqrt{a}x$ by $x$ and $ac$ by $c$ if $\ep_1=-1$
and $\sqrt{a}y$ by $y$ and $ad$ by $d$ if $\ep_2=-1$,
we may assume that $\ep_1=\ep_2=1$.
By Lemma \ref{l2.2}, we have $k(\sqrt{a},x,y)^{\langle -I \rangle}=k(\sqrt{a},u,v)$ where
\[
u=\frac{y(x^2-c)}{x^2y^2-cd},\quad 
v=\frac{x(y^2-d)}{x^2y^2-cd}
\]
and the action of $\lambda$ on $k(\sqrt{a},u,v)$ is given by 
\[ \lambda: \ \sqrt{a}\mapsto -\sqrt{a},\ u\mapsto \frac{u}{du^2-cv^2},\ v\mapsto -\frac{v}{du^2-cv^2}. \]
Define $U:=\sqrt{a}u/v$, $V:=(du^2-cv^2)/v$. 
We find that 
$u=\sqrt{a}UV/(dU^2-ac)$, $v=aV/(dU^2-ac)$. 
Then we have 
$k(\sqrt{a},u,v)=k(\sqrt{a},U,V)$ and 
\[ \lambda: \ \sqrt{a}\mapsto -\sqrt{a}, \ U\mapsto U, \ V\mapsto \frac{-(d/a)(U^2-ac/d)}{V}. \]
It follows from Theorem \ref{t2.6} that
$K(x,y)^G$ is $k$-rational if and only if 
$K(x,y)^G$ is $k$-unirational if and only if 
$(a,-d/a)_{2,k(\sqrt{cd})}=0$ if and only if 
$(a,d)_{2,k(\sqrt{cd})}=0$ 
because $(a,-d/a)_{2,k(\sqrt{cd})}=(a,d)_{2,k(\sqrt{cd})}+(a,-1/a)_{2,k(\sqrt{cd})}$ 
and $(a,-1/a)_{2,k(\sqrt{cd})}=(a,-a)_{2,k(\sqrt{cd})}=0$ (see Lemma \ref{l2.5}).\\

\noindent
\textbf{Case 3:} $H=\langle \lambda \rangle$.
We have $K=k(\sqrt{a})$ for some $a\in k$.
Then the actions of $\lambda$ and $-I$ on $k(\sqrt{a})(x,y)$ are given by 
\begin{align*}
\lambda : \sqrt{a}\mapsto \sqrt{a}, \ x\mapsto \ep_1x, \ y\mapsto \frac{d}{y},\quad 
-I : \sqrt{a}\mapsto -\sqrt{a}, \ x\mapsto \frac{c}{x}, \ y\mapsto \frac{\ep_2d}{y}.
\end{align*}
By replacing $\sqrt{a}y$ by $y$ and $ad$ by $d$, we may assume that $\ep_2=1$.\\

\noindent
\textbf{Case 3-1:} $H=\langle \lambda \rangle$ and $\ep_1=1$.
We have $k(\sqrt{a},x,y)^{\langle\lambda\rangle}=k(\sqrt{a},x,Y)$ where $Y=y+\frac{d}{y}$,
and the action of $-I$ on $k(\sqrt{a},x,Y)$ is given by 
\[
-I: \ \sqrt{a}\mapsto -\sqrt{a}, \ x\mapsto \frac{c}{x}, \ Y\mapsto Y.
\]
It follows from Theorem \ref{t2.6} (Lemma \ref{l2.5}) that
$K(x,y)^G$ is $k$-rational if and only if 
$K(x,y)^G$ is $k$-unirational if and only if 
$(a,c)_{2,k}=0$.\\

\noindent
\textbf{Case 3-2:} $H=\langle \lambda \rangle$ and $\ep_1=-1$.
We have $k(\sqrt{a},x,y)^{\langle\lambda\rangle}=k(\sqrt{a},u,v)$ where 
$u=\frac{1}{2}(y+\frac{d}{y})$, $v=\frac{1}{2x}(y-\frac{d}{y})$ and the action of $-I$ on $k(\sqrt{a},u,v)$ is given by 
\[
-I: \ \sqrt{a}\mapsto -\sqrt{a}, \ u\mapsto u,\ v\mapsto \frac{-(1/c)(u^2-d)}{v}.
\]
It follows from Theorem \ref{t2.6} that
$K(x,y)^G$ is $k$-rational if and only if 
$K(x,y)^G$ is $k$-unirational if and only if 
$(a,-c)_{2,k(\sqrt{ad})}=0$.\\

\noindent
\textbf{Case 4:} $H=\langle -\lambda \rangle$.
We have $K=k(\sqrt{a})$ for some $a\in k$.
Then the actions of $-\lambda$ and $-I$ on $k(\sqrt{a})(x,y)$ are given by 
\begin{align*}
-\lambda : \sqrt{a}\mapsto \sqrt{a}, \ x\mapsto \frac{c}{x}, \ y\mapsto \ep_2 y,\quad 
-I : \sqrt{a}\mapsto -\sqrt{a}, \ x\mapsto \frac{\ep_1c}{x}, \ y\mapsto \frac{d}{y}.
\end{align*}
After changing $x\leftrightarrow y$, $c\leftrightarrow d$, 
$\ep_1\leftrightarrow \ep_2$, this action becomes the same 
as in Case 3. 
In particular, we may assume that $\ep_1=1$. 
By the result of Case 3, we get: 

When $\ep_2=1$, 
$K(x,y)^G$ is $k$-rational if and only if 
$K(x,y)^G$ is $k$-unirational if and only if 
$(a,d)_{2,k}=0$. 
When $\ep_2=-1$, 
$K(x,y)^G$ is $k$-rational if and only if 
$K(x,y)^G$ is $k$-unirational if and only if 
$(a,-d)_{2,k(\sqrt{ac})}=0$.

\subsection{The case of $G=V_4^{(2)}=\langle \tau, -I \rangle$}\label{ss38}

The action of $G$ on $K(x,y)$ is given by
\[
\tau : x\mapsto by, \ y\mapsto \frac{x}{b},\quad
-I : x\mapsto \frac{c}{x}, \ y\mapsto \frac{d}{y}\quad
(c,d\in k\setminus\{0\}). 
\]
By replacing $by$ by $y$, we may assume that $b=1$.
By the equalities $\tau^2=I$ and
$\tau(-I)=(-I)\tau$,
we see that the action of $G$ on $K(x,y)$ is given by
\[
\tau : \ x\mapsto y, \ y\mapsto x,\quad
-I : x\mapsto \frac{c}{x}, \ y\mapsto \frac{c}{y}\quad
(c\in k\setminus\{0\}).
\]

We will treat the problem
for each of the proper normal subgroups 
$H=\{1\}$, $\langle -I \rangle$, $\langle \tau \rangle$, 
$\langle -\tau \rangle$ of $G$.\\

\noindent
\textbf{Case 1:} $H=\{1\}$.
We have $K=k(\sqrt{a},\sqrt{b})$ for some $a,b\in k$.
Then the group $G$ acts on $K(x,y)$ by
\begin{align*}
\tau : \sqrt{a}\mapsto -\sqrt{a}, \ \sqrt{b}\mapsto \sqrt{b}, \ x\mapsto y, \ y\mapsto x,\quad 
-I : \sqrt{a}\mapsto \sqrt{a}, \ \sqrt{b}\mapsto -\sqrt{b}, \ x\mapsto \frac{c}{x}, \ y\mapsto \frac{c}{y}. 
\end{align*}
We have $K(x,y)^{\langle \tau \rangle}=k(\sqrt{b})(u,v)$ where
\[
u=\frac{1}{2}(x+y),\quad v=\frac{1}{2\sqrt{a}}(x-y)
\]
and the action of $-I$ on $k(\sqrt{b})(u,v)$ is given by
\[
-I : \sqrt{b}\mapsto -\sqrt{b}, \ u\mapsto \frac{cu}{u^2-av^2}, \ v\mapsto -\frac{cv}{u^2-av^2}.
\]
Define $U:=\sqrt{b}u/v$, $V:=(u^2-av^2)/v$. 
We find that  
$u=\sqrt{b}UV/(U^2-ab)$, $v=bV/(U^2-ab)$. 
Then we have 
$k(\sqrt{b})(u,v)=k(\sqrt{b})(U,V)$ and 
\[
-I : \sqrt{b}\mapsto -\sqrt{b}, \ U\mapsto U, \ V\mapsto \frac{-(c/b)(U^2-ab)}{V}.
\]
It follows from Theorem \ref{t2.6} that
$K(x,y)^G$ is $k$-rational if and only if 
$K(x,y)^G$ is $k$-unirational if and only if 
$(b,-c/b)_{2,k(\sqrt{a})}=0$ if and only if 
$(b,c)_{2,k(\sqrt{a})}=0$ 
because $(b,-c/b)_{2,k(\sqrt{a})}=(b,c)_{2,k(\sqrt{a})}+(b,-1/b)_{2,k(\sqrt{a})}$ 
and $(b,-1/b)_{2,k(\sqrt{a})}=(b,-b)_{2,k(\sqrt{a})}=0$ (see Lemma \ref{l2.5}).\\

\noindent
\textbf{Case 2:} $H=\langle -I \rangle$.
We have $K=k(\sqrt{a})$ for some $a\in k$.
Then the group $G$ acts on $K(x,y)$ by
\begin{align*}
\tau : \sqrt{a}\mapsto -\sqrt{a}, \ x\mapsto y, \ y\mapsto x,\quad 
-I : \sqrt{a}\mapsto \sqrt{a}, \ x\mapsto \frac{c}{x}, \ y\mapsto \frac{c}{y}.
\end{align*}
By Lemma \ref{l2.1},
we have $K(x,y)^{\langle -I \rangle}=K(X,Y)$ where
\[
X=\frac{xy+c}{x+y},\quad Y=\frac{xy-c}{x-y}
\]
and the action of $\tau$ on $K(X,Y)$ is given by
\begin{align*}
\tau : \sqrt{a}\mapsto -\sqrt{a}, \ X\mapsto X, \ Y\mapsto -Y.
\end{align*}
Hence $K(x,y)^G=k(X,\sqrt{a}Y)$ is $k$-rational.\\

\noindent
\textbf{Case 3:} $H=\langle \tau \rangle$.
We have $K=k(\sqrt{a})$ for some $a\in k$.
Then the group $G$ acts on $K(x,y)$ by
\begin{align*}
\tau : \sqrt{a}\mapsto \sqrt{a}, \ x\mapsto y, \ y\mapsto x,\quad 
-I : \sqrt{a}\mapsto -\sqrt{a}, \ x\mapsto \frac{c}{x}, \ y\mapsto \frac{c}{y}.
\end{align*}
We have $K(x,y)^{\langle \tau \rangle}=K(u,v)$ where
\[ u=\frac{xy}{c},\quad v=x+y \]
and the action of $-I$ on $K(u,v)$ is given by
\begin{align*}
-I &: \sqrt{a}\mapsto -\sqrt{a}, \ u\mapsto \frac{1}{u}, \ v\mapsto \frac{v}{u}.
\end{align*}
By Theorem \ref{thHKY14}, $K(x,y)^G$ is $k$-rational.\\

\noindent
\textbf{Case 4:} $H=\langle -\tau \rangle$.
We have $K=k(\sqrt{a})$ for some $a\in k$.
Then the group $G$ acts on $K(x,y)$ by
\begin{align*}
\tau : \sqrt{a}\mapsto -\sqrt{a}, \ x\mapsto y, \ y\mapsto x,\quad 
-I : \sqrt{a}\mapsto -\sqrt{a}, \ x\mapsto \frac{c}{x}, \ y\mapsto \frac{c}{y}.
\end{align*}
We have $K(x,y)^{\langle \tau \rangle}=k(u,v)$ where
\[ u=\frac{x+y}{\sqrt{a}(x-y)},\quad v=\sqrt{a}(x-y) \]
and the action of $-I$ on $k(u,v)$ is given by
\begin{align*}
-I : u\mapsto u, \ v\mapsto \frac{4ac}{(au^2-1)v}.
\end{align*}
Hence $K(x,y)^G=k(u,v)^{\langle -I \rangle}=k(u,v+\frac{4ac}{(au^2-1)v})$
is $k$-rational.

\subsection{The case of $G=S_3^{(1)}=\langle \rho^2, \tau \rangle$}\label{ss39}

The action of $G$ on $K(x,y)$ is given by
\begin{align*}
\rho^2 :  \ x\mapsto by, \ y\mapsto\frac{c}{xy},\quad
\tau : \ x\mapsto dy, \ y\mapsto ex\quad
(b,c,d,e\in k\setminus\{0\}).
\end{align*}
By replacing $by$ by $y$, $b^2c$ by $c$, 
$d/b$ by $d$ and $be$ by $e$, we may assume that $b=1$. 
By the equalities $\tau^2=I$ and
$\tau(\rho^2)\tau^{-1}=(\rho^2)^2$,
we see that the action of $G$ on $K(x,y)$ is given by
\begin{align*}
\rho^2 :  \ x\mapsto y, \ y\mapsto\frac{c}{xy},\quad
\tau :  \ x\mapsto y, \ y\mapsto x\quad
(c\in k\setminus\{0\}).
\end{align*}

We will treat the problem
for each of the proper normal subgroups
$H=\{1\}$, $\langle \rho^2 \rangle$ of $G$.\\

\textbf{Case 1:} $H=\{1\}$. In this case, $[K:k]=6$ and we consider the field $F=K^{\langle\rho^2\rangle}$ 
with $[F:k]=2$. 
We find that $K(x,y)^G\otimes_kF=(K(x,y)^G)F=K(x,y)^{\langle\rho^2\rangle}$ 
because $F$ $\cap$ $K(x,y)^G=k$ and $[K(x,y)^{\langle\rho^2\rangle}:K(x,y)^G]=2$.\\

\textbf{Case 1-1:} $H=\{1\}$ and $\omega\in k$. 
We assume that $\omega\in k$ where 
$\omega$ is a primitive cubic root of unity in $\overline{k}$. 
Then we have 
$K=F(\sqrt[3]{\alpha})$ for some $\alpha\in F=K^{\langle\rho^2\rangle}$ with $[F:k]=2$. 
We may also assume that $\alpha\in F\setminus k$ without loss of generality. 
Then we have $F=k(\alpha)$.\\

(i) Suppose that $(\alpha,c)_{3,k(\alpha)}=0$. 
Then there exists 
$\gamma_0\in F(\sqrt[3]{\alpha})$ such that 
$c=
N_{F(\sqrt[3]{\alpha})/F}(\gamma_0)=\gamma_0\gamma_1\gamma_2$ 
where $\gamma_1=\rho^2(\gamma_0)$, $\gamma_2=\rho^2(\gamma_1)$. 
We may assume that $\gamma_0\in K^{\langle\tau\rangle}$ without loss of generality. 
Then it follows from $\tau\rho^2\tau^{-1}=(\rho^2)^2$ that 
$\tau: \gamma_0\mapsto \gamma_0$, $\gamma_1\mapsto \gamma_2$, 
$\gamma_2\mapsto\gamma_1$. 
Define 
\begin{align*}
X:=\frac{x}{\gamma_1},\quad 
Y:=\frac{y}{\gamma_2}.
\end{align*}
Then $K(x,y)=K(X,Y)$ and the actions of $\rho^2$ and $\tau$ on 
$K(X,Y)=F(\sqrt[3]{\alpha})(X,Y)$ are given by
\begin{align*}
\rho^2 &: \sqrt[3]{\alpha}\mapsto \omega\,\sqrt[3]{\alpha},
\ \sqrt[3]{\beta}\mapsto\omega^{-1}\sqrt[3]{\beta},
\ X\mapsto Y, \ Y\mapsto\frac{1}{XY},\\
\tau &: \sqrt[3]{\alpha}\mapsto\sqrt[3]{\beta},
\ \sqrt[3]{\beta}\mapsto\sqrt[3]{\alpha},
\ X\mapsto Y, \ Y\mapsto X
\end{align*}
for 
$\beta=\tau(\alpha)\in F$ which satisfies  
$K=F(\sqrt[3]{\alpha})=F(\sqrt[3]{\beta})$. 
%
Hence, by Voskresenskii's theorem (Theorem \ref{thVos67}),  
$K(x,y)^G=F(\sqrt[3]{\alpha})(X,Y)^{\langle\rho^2,\tau\rangle}$ is $k$-rational. 

Indeed, we define
\begin{align*}
u:=\frac{1+\omega^{-1}X+\omega XY}{1+X+XY},\quad
v:=\frac{1+\omega X+\omega^{-1}XY}{1+X+XY}.
\end{align*}
Then it follows from Lemma \ref{l2.4} that
$K(X,Y)=K(u,v)$ and
\begin{align*}
\rho^2 &: \sqrt[3]{\alpha}\mapsto \omega\,\sqrt[3]{\alpha},
\ \sqrt[3]{\beta}\mapsto\omega^{-1}\sqrt[3]{\beta},
\ u\mapsto \omega u,\ v\mapsto \omega^{-1}v,\\
\tau &: \sqrt[3]{\alpha}\mapsto\sqrt[3]{\beta},
\ \sqrt[3]{\beta}\mapsto\sqrt[3]{\alpha},
\ u\mapsto \omega\,\frac{v-u^2}{uv-1},\ 
v\mapsto \omega^{-1}\,\frac{u-v^2}{uv-1}.
\end{align*}
Define
\begin{align*}
s:=\omega^{-1}\frac{u-v^2}{v(uv-1)},\quad 
t:=\omega^{-1}\frac{u(uv-1)}{v-u^2}.
\end{align*}
We find that 
$u=\omega^{-1} (s-t)t/(st^2-1)$, $v=-\omega(s-t)/(s^2t-1)$. 
Then $K(u,v)=K(s,t)$ and
\begin{align*}
\rho^2 &: \sqrt[3]{\alpha}\mapsto \omega\,\sqrt[3]{\alpha},
\ \sqrt[3]{\beta}\mapsto\omega^{-1}\sqrt[3]{\beta},
\ s\mapsto \omega^{-1}\,s, \ t\mapsto\omega^{-1}\,t,\\
\tau &: \sqrt[3]{\alpha}\mapsto\sqrt[3]{\beta},
\ \sqrt[3]{\beta}\mapsto\sqrt[3]{\alpha},
\ s \mapsto\frac{1}{s},\ t\mapsto\frac{1}{t}.
\end{align*}
Define
\begin{align*}
S:=\frac{\sqrt[3]{\beta}}{\sqrt[3]{\alpha}}\ s,\quad 
T:=\frac{\sqrt[3]{\beta}}{\sqrt[3]{\alpha}}\ t. 
\end{align*}
Then we have $K(s,t)=K(S,T)$ and
\begin{align*}
\rho^2 &: \sqrt[3]{\alpha}\mapsto \omega\,\sqrt[3]{\alpha},
\ \sqrt[3]{\beta}\mapsto\omega^{-1}\sqrt[3]{\beta},
\ S\mapsto S,\ T\mapsto T,\\
\tau &: \sqrt[3]{\alpha}\mapsto\sqrt[3]{\beta},
\ \sqrt[3]{\beta}\mapsto\sqrt[3]{\alpha},
\ S\mapsto \frac{1}{S}, \ T\mapsto \frac{1}{T}.
\end{align*}
Hence $K(x,y)^G=(K(S,T)^{\langle\rho^2\rangle})^{\langle\tau\rangle}=k(\alpha)(S,T)^{\langle\tau\rangle}
=k((\alpha-\beta)\frac{S+1}{S-1},(\alpha-\beta)\frac{T+1}{T-1})$ is $k$-rational.\\ 

(ii) Suppose that $(\alpha,c)_{3,k(\alpha)}\neq 0$. 
Then by Case 1-1 of Section \ref{ss34}, 
$K(x,y)^G\otimes_kF=K(x,y)^{\langle\rho^2\rangle}$ is not $F$-unirational 
where $F=K^{\langle\rho^2\rangle}=k(\alpha)$ with $[F:k]=2$. 
Hence $K(x,y)^G$ is not $k$-unirational.\\ 


\noindent
\textbf{Case 1-2:} $H=\{1\}$ and $\omega\in K\setminus k$.
In this case, we have $F=K^{\langle\rho^2\rangle}=k(\omega)$ with $[F:k]=2$
and $K=F(\sqrt[3]{\alpha})$ for some $\alpha\in F=k(\omega)$. 
We may also assume that $\alpha\in F\setminus k$ without loss of generality. 
Then we have $F=k(\omega)=k(\alpha)$.\\

(i) Suppose that $(\alpha,c)_{3,k(\alpha)}=0$. 
Then there exists 
$\gamma_0\in F(\sqrt[3]{\alpha})$ such that 
$c=
N_{F(\sqrt[3]{\alpha})/F}(\gamma_0)=\gamma_0\gamma_1\gamma_2$ 
where $\gamma_1=\rho^2(\gamma_0)$, $\gamma_2=\rho^2(\gamma_1)$. 
We may assume that $\gamma_0\in K^{\langle\tau\rangle}$ without loss of generality. 
Then it follows from $\tau\rho^2\tau^{-1}=(\rho^2)^2$ that 
$\tau: \gamma_0\mapsto \gamma_0$, $\gamma_1\mapsto \gamma_2$, 
$\gamma_2\mapsto\gamma_1$. 
Define 
\begin{align*}
X:=\frac{x}{\gamma_1},\quad 
Y:=\frac{y}{\gamma_2}.
\end{align*}
Then the actions of $\rho^2$ and $\tau$
on $K(X,Y)=F(\sqrt[3]{\alpha})(X,Y)=k(\omega,\sqrt[3]{\alpha})(X,Y)$ are given by
\begin{align*}
\rho^2 &: \sqrt[3]{\alpha}\mapsto \omega\,\sqrt[3]{\alpha},
\ \omega\mapsto\omega,\ X\mapsto Y, \ Y\mapsto\frac{1}{XY},\\
\tau &: \sqrt[3]{\alpha}\mapsto\sqrt[3]{\alpha},
\ \omega\mapsto\omega^{-1},
\ X\mapsto Y, \ Y\mapsto X. 
\end{align*}
%
Hence, by Voskresenskii's theorem (Theorem \ref{thVos67}),  
$K(x,y)^G=k(\omega,\sqrt[3]{\alpha})(X,Y)^{\langle\rho^2,\tau\rangle}$ is $k$-rational.\\ 

(ii) Suppose that $(\alpha,c)_{3,k(\alpha)}\neq 0$. 
Then by Case 1-1 of Section \ref{ss34}, 
$K(x,y)^G\otimes_kF=K(x,y)^{\langle\rho^2\rangle}$ is not $F$-unirational 
where $F=k(\omega)=k(\alpha)$ with $[F:k]=2$. 
Hence $K(x,y)^G$ is not $k$-unirational.\\ 

\noindent
\textbf{Case 1-3:} $H=\{1\}$ and $\omega\not\in K$.
%
We take $L=K(\omega)$
and ${\rm Gal}(L/K)=\langle\varphi_\omega\rangle$
where 
$\varphi_\omega : \omega\mapsto\omega^{-1}$.
We extend the actions of $\rho^2$, $\tau$, $\varphi_\omega$ to $L(x,y)$ in the natural way.
Then $\langle\rho^2,\tau,\varphi_\omega\rangle\simeq S_3\times C_2\simeq D_6$ 
and we find that $L(x,y)^{\langle\rho^2,\tau,\varphi_\omega\rangle}=L^{\langle\varphi_\omega\rangle}(x,y)^G=K(x,y)^G$. 

We also have 
$L=F(\omega,\sqrt[3]{\alpha})$ for some $\alpha\in F(\omega)$ and $F=K^{\langle\rho^2\rangle}$ with $[F:k]=2$. 
We may assume that 
$\alpha\in F(\omega)\setminus k$ without loss of generality. 
Then we have $F(\omega)=K^{\langle\rho^2\rangle}(\omega)=k(\omega,\alpha)$ with $[F(\omega):k]=4$.\\  

(i) Suppose that $(\alpha,c)_{3,k(\omega,\alpha)}=0$. 
Then we see that there exists 
$\gamma_0\in L=F(\omega,\sqrt[3]{\alpha})$ such that 
$c=N_{F(\omega,\sqrt[3]{\alpha})/F(\omega)}(\gamma_0)$ $=$ $\gamma_0\gamma_1\gamma_2$ 
where $\gamma_1=\rho^2(\gamma_0)$, $\gamma_2=\rho^2(\gamma_1)$. 
We may assume that $\gamma_0\in K^{\langle\tau\rangle}$ without loss of generality. 
Then it follows from $\tau\rho^2\tau^{-1}=(\rho^2)^2$ that 
$\tau: \gamma_0\mapsto \gamma_0$, $\gamma_1\mapsto \gamma_2$, 
$\gamma_2\mapsto\gamma_1$. 
Define 
\begin{align*}
X:=\frac{x}{\gamma_1},\quad 
Y:=\frac{y}{\gamma_2}.
\end{align*}
Then $L(x,y)=L(X,Y)$ and the actions of $\rho^2$, $\tau$ and $\varphi_\omega$ on 
$L(X,Y)=F(\omega,\sqrt[3]{\alpha})(X,Y)$ are given by
\begin{align*}
\rho^2 &: \sqrt[3]{\alpha}\mapsto \omega\,\sqrt[3]{\alpha},
\ \sqrt[3]{\beta}\mapsto\omega^{-1}\sqrt[3]{\beta},\ \omega\mapsto\omega,
\ X\mapsto Y, \ Y\mapsto\frac{1}{XY},\\
\tau &: \sqrt[3]{\alpha}\mapsto\sqrt[3]{\beta},
\ \sqrt[3]{\beta}\mapsto\sqrt[3]{\alpha},\ \omega\mapsto\omega,
\ X\mapsto Y, \ Y\mapsto X,\\ 
\varphi_\omega &:  \sqrt[3]{\alpha}\mapsto\sqrt[3]{\beta},\ 
\sqrt[3]{\beta}\mapsto \sqrt[3]{\alpha},\ 
\omega\mapsto\omega^{-1},
\ X\mapsto X, \ Y\mapsto Y
\end{align*}
for 
$\beta=\tau(\alpha)\in F$ which satisfies  
$L=F(\omega,\sqrt[3]{\alpha})=F(\omega,\sqrt[3]{\beta})$. 
Define
\begin{align*}
u:=\frac{1+\omega^{-1}X+\omega XY}{1+X+XY},\quad
v:=\frac{1+\omega X+\omega^{-1}XY}{1+X+XY}.
\end{align*}
Then it follows from Lemma \ref{l2.4} that
$L(X,Y)=L(u,v)$ and
\begin{align*}
\rho^2 &: \sqrt[3]{\alpha}\mapsto \omega\,\sqrt[3]{\alpha},
\ \sqrt[3]{\beta}\mapsto\omega^{-1}\sqrt[3]{\beta},\ \omega\mapsto\omega,
\ u\mapsto \omega u,\ v\mapsto \omega^{-1}v,\\
\tau &: \sqrt[3]{\alpha}\mapsto\sqrt[3]{\beta},
\ \sqrt[3]{\beta}\mapsto\sqrt[3]{\alpha},\ \omega\mapsto\omega,
\ u\mapsto \omega\,\frac{v-u^2}{uv-1},\ 
v\mapsto \omega^{-1}\,\frac{u-v^2}{uv-1},\\
\varphi_\omega &:  \sqrt[3]{\alpha}\mapsto\sqrt[3]{\beta},\ 
\sqrt[3]{\beta}\mapsto \sqrt[3]{\alpha},\ 
\omega\mapsto\omega^{-1},
\ u\mapsto v, \ v\mapsto u. 
\end{align*}
As in Case 1-1, define
\begin{align*}
s:=\omega^{-1}\frac{u-v^2}{v(uv-1)},\quad 
t:=\omega^{-1}\frac{u(uv-1)}{v-u^2}.
\end{align*}
Then $L(u,v)=L(s,t)$ and
\begin{align*}
\rho^2 &: \sqrt[3]{\alpha}\mapsto \omega\,\sqrt[3]{\alpha},
\ \sqrt[3]{\beta}\mapsto\omega^{-1}\sqrt[3]{\beta},\ \omega\mapsto\omega,
\ s\mapsto \omega^{-1}\,s, \ t\mapsto\omega^{-1}\,t,\\
\tau &: \sqrt[3]{\alpha}\mapsto\sqrt[3]{\beta},
\ \sqrt[3]{\beta}\mapsto\sqrt[3]{\alpha},\ \omega\mapsto\omega,
\ s \mapsto\frac{1}{s},\ t\mapsto\frac{1}{t},\\
\varphi_\omega &:  \sqrt[3]{\alpha}\mapsto\sqrt[3]{\beta},\ 
\sqrt[3]{\beta}\mapsto \sqrt[3]{\alpha},\ 
\omega\mapsto\omega^{-1},
\ s\mapsto \frac{1}{t}, \ t\mapsto \frac{1}{s}.
\end{align*}
Define
\begin{align*}
S:=\frac{\sqrt[3]{\beta}}{\sqrt[3]{\alpha}}\ s,\quad  
T:=\frac{\sqrt[3]{\beta}}{\sqrt[3]{\alpha}}\ t. 
\end{align*}
Then we have $L(s,t)=L(S,T)$ and
\begin{align*}
\rho^2 &: \sqrt[3]{\alpha}\mapsto \omega\,\sqrt[3]{\alpha},
\ \sqrt[3]{\beta}\mapsto\omega^{-1}\sqrt[3]{\beta},\ \omega\mapsto\omega,
\ S\mapsto S,\ T\mapsto T,\\
\tau &: \sqrt[3]{\alpha}\mapsto\sqrt[3]{\beta},
\ \sqrt[3]{\beta}\mapsto\sqrt[3]{\alpha},\ \omega\mapsto\omega,
\ S\mapsto \frac{1}{S}, \ T\mapsto \frac{1}{T},\\
\varphi_\omega &:  \sqrt[3]{\alpha}\mapsto\sqrt[3]{\beta},\ 
\sqrt[3]{\beta}\mapsto \sqrt[3]{\alpha},\ 
\omega\mapsto\omega^{-1},
\ S\mapsto \frac{1}{T}, \ T\mapsto \frac{1}{S}. 
\end{align*}
Hence it follows from Voskresenskii's theorem (Theorem \ref{thVos67}) that 
$K(x,y)^G=L(S,T)^{\langle\rho^2,\tau,\varphi_\omega\rangle}=L^{\langle\rho^2\rangle}(S,T)^{\langle\tau,\varphi_\omega\rangle}
=F(\omega)(S,T)^{\langle\tau,\varphi_\omega\rangle}$ is $k$-rational.\\

(ii) Suppose that $(\alpha,c)_{3,k(\omega,\alpha)}\neq 0$. 
By Case 1-2 of Section \ref{ss34}, 
$K(x,y)^G\otimes_kF=K(x,y)^{\langle\rho^2\rangle}$ is not $F$-unirational 
where $F(\omega)=k(\omega,\alpha)$ and 
$F=K^{\langle\rho^2\rangle}$ with $[F:k]=2$. 
Hence $K(x,y)^G$ is not $k$-unirational. 

\begin{remark}\label{r3.3}
One of the referees told that 
the case $G=S_3^{(1)}$ with $H=\{1\}$ corresponds to 
a del Pezzo surface $S$ of degree $6$ and we see that 
$S$ is $k$-rational if and only if $S_F=S\otimes_k F$ is $F$-rational 
where $F=K^{\langle\rho^2\rangle}$ with $[F:k]=2$. 
It is also interesting to compare this with Trepalin \cite[Lemma 4]{Tr1}. 
\end{remark}

\noindent
\textbf{Case 2:} $H=\langle \rho^2 \rangle$.
We have $K=k(\sqrt{a})$ for some $a\in k$. 
The group $G$ acts on $K(x,y)$ by
\begin{align*}
\rho^2 : \sqrt{a}\mapsto \sqrt{a}, \ x\mapsto y, \ y\mapsto\frac{c}{xy},\quad 
\tau : \sqrt{a}\mapsto -\sqrt{a}, \ x\mapsto y, \ y\mapsto x.
\end{align*}
By Lemma \ref{l2.3},
we have $K(x,y)^{\langle \rho^2 \rangle}=K(X,Y)$ where
\begin{align*}
X &= \frac{y(y^3x^3+cx^3-3cyx^2+c^2)}{y^2x^4-y^3x^3+y^4x^2-cyx^2-cy^2x+c^2},\\
Y &= \frac{x(x^3y^3+cy^3-3cxy^2+c^2)}{y^2x^4-y^3x^3+y^4x^2-cyx^2-cy^2x+c^2}
\end{align*}
and the action of $\tau$ on $K(X,Y)$ is given by
\[
\tau : \sqrt{a}\mapsto -\sqrt{a}, \ X\mapsto Y, \ Y\mapsto X.
\]
We have
$K(x,y)^G=(k(\sqrt{a})(x,y)^{\langle\rho^2\rangle})^{\langle\tau\rangle}=
k(X+Y,\sqrt{a}(X-Y))$.
Hence $K(x,y)^G$ is $k$-rational.

\subsection{The case of $G=S_3^{(2)}=\langle \rho^2, -\tau \rangle$}\label{ss310}


The action of $G$ on $K(x,y)$ is given by
\begin{align*}
\rho^2 :  \ x\mapsto by, \ y\mapsto\frac{c}{xy},\quad
-\tau :  \ x\mapsto \frac{d}{y}, \ y\mapsto \frac{e}{x}\quad
(b,c,d,e\in k\setminus\{0\}).
\end{align*}
By replacing $by$ by $y$, $b^2c$ by $c$, 
$bd$ by $d$ and $be$ by $e$, 
we may assume that $b=1$. 
By the equalities $(-\tau)^2=I$ and
$(-\tau)(\rho^2)(-\tau)^{-1}=(\rho^2)^2$,
we see that the action of $G$ on $K(x,y)$ is given by
\begin{align*}
\rho^2 :  \ x\mapsto y, \ y\mapsto\frac{c}{xy},\quad
-\tau :  \ x\mapsto \frac{d}{y}, \ y\mapsto \frac{d}{x}
\end{align*}
with $c^2=d^3$.
Define $X:=dx/c$, $Y:=dy/c$.
Then $K(x,y)=K(X,Y)$ and the action of $G$ on $K(X,Y)$ is given by
\begin{align*}
\rho^2 :  \ X\mapsto Y, \ Y\mapsto\frac{1}{XY},\quad
-\tau :  \ X\mapsto \frac{1}{Y}, \ Y\mapsto \frac{1}{X}.
\end{align*}
Hence the problem can be reduced to the purely quasi-monomial case.
By Theorem \ref{thHKY14},
$K(x,y)^G$ is $k$-rational for each of the proper normal subgroups 
$H=\{1\}$, $\langle \rho^2 \rangle$ of $G$.

\subsection{The case of $G=D_4=\langle \sigma, \tau \rangle$}\label{ss311}

By the equalities $\sigma^4=I$ and
$\tau\sigma\tau^{-1}=\sigma^{-1}$,
we see that the action of $G$ on $K(x,y)$ is given by
\begin{align*}
\sigma &: x\mapsto y, \ y\mapsto \frac{c}{x}\quad
(c\in k\setminus\{0\}), \\
\tau &: x\mapsto \ep y, \ y\mapsto \ep x\quad (\ep=\pm 1).
\end{align*}

We will treat the problem
for each of the proper normal subgroups 
$H=\{1\}$, $\langle -I \rangle=\langle\sigma^2\rangle$, 
$\langle -I, \tau\sigma \rangle=\langle -I,\lambda\rangle$, 
$\langle -I, \tau \rangle$, $\langle \sigma \rangle$ of $G$.\\

\noindent
\textbf{Case 1:} $H=\{1\}$. 
%
In this case, $[K:k]=8$ and 
we can take suitable $\alpha, \beta\in K$ which satisfy
$K=k(\alpha,\beta)$.
Then the action of $G$ on $K(x,y)$ is given by
\begin{align*}
\sigma &: \alpha\mapsto\beta,\ \beta\mapsto-\alpha,\ 
x\mapsto y, \ y\mapsto \frac{c}{x},\\
\tau &: \alpha\mapsto\beta,\ \beta\mapsto\alpha,\ 
x\mapsto \ep y, \ y\mapsto \ep x,\\
\sigma\tau &: \alpha\mapsto-\alpha,\ \beta\mapsto\beta,\ 
x\mapsto \frac{c}{\ep x}, \ y\mapsto \ep y,\\
\sigma^3\tau &: \alpha\mapsto\alpha,\ \beta\mapsto-\beta,\ 
x\mapsto \ep x, \ y\mapsto \frac{c}{\ep y}.
\end{align*} 
We consider the field 
$F=K^{\langle\sigma^2,\sigma\tau\rangle}=K^{\langle\sigma\tau,\sigma^3\tau\rangle}
=k(\alpha^2,\beta^2)=k(\alpha^2)$ with $[F:k]=2$. 
%
We find that $K(x,y)^G\otimes_k F=(K(x,y)^G)F=K(x,y)^{\langle\sigma^2,\sigma\tau\rangle}$ 
because $F$ $\cap$ $K(x,y)^G=k$ and $[K(x,y)^{\langle\sigma^2,\sigma\tau\rangle}:K(x,y)^G]=2$.\\ 

\noindent
\textbf{Case 1-1:} $H=\{1\}$ and $\ep=1$.\\ 

(i) Suppose $(\alpha^2,c)_{2,k(\alpha^2)}=0$. 
Then there exists 
$\gamma_0\in k(\alpha)$ such that 
$c=
N_{k(\alpha)/k(\alpha^2)}(\gamma_0)=\gamma_0\gamma_1$ 
where $\gamma_1=\sigma\tau(\gamma_0)$. 
Define 
\begin{align*}
X:=\frac{x}{\gamma_0},\quad 
Y:=\frac{y}{\sigma(\gamma_0)}.
\end{align*}
Then $K(x,y)=K(X,Y)$ and the actions of 
$\sigma$, $\sigma\tau$ and $\sigma^3\tau$ on $K(X,Y)$ are given by 
\begin{align*}
\sigma : \alpha\mapsto \beta,\ \beta\mapsto -\alpha,
\ X\mapsto Y, \ Y\mapsto \frac{1}{X},\\ 
\sigma\tau : \alpha\mapsto -\alpha,\ \beta\mapsto \beta,
\ X\mapsto \frac{1}{X}, \ Y\mapsto Y,\\
\sigma^3\tau : \alpha\mapsto \alpha,\ \beta\mapsto -\beta,
\ X\mapsto X, \ Y\mapsto \frac{1}{Y}.
\end{align*}
Hence, by Voskresenskii's theorem (Theorem \ref{thVos67}), 
$K(x,y)^G=K(X,Y)^{\langle\sigma,\sigma\tau\rangle}$ is $k$-rational.

Indeed, we define
\[
u:=\alpha\,\frac{X+1}{X-1},\quad 
v:=\beta\,\frac{Y+1}{Y-1}. 
\]
Then the actions of $\sigma$, $\sigma\tau$ and $\sigma^3\tau$ 
on $K(x,y)=k(\alpha,\beta)(u,v)$ are given by 
\begin{align*}
\sigma &: \alpha\mapsto \beta,\ \beta\mapsto -\alpha, 
\ u\mapsto v, \ v\mapsto u, \\
\sigma\tau &: \alpha\mapsto -\alpha,\ \beta\mapsto \beta,
\ u\mapsto u,\ v\mapsto v,\\ 
\sigma^3\tau &: \alpha\mapsto \alpha,\ \beta\mapsto -\beta,
\ u\mapsto u,\ v\mapsto v. 
\end{align*}
We have $K(x,y)^{\langle\sigma\tau,\sigma^3\tau\rangle}=k(\alpha,\beta)(u,v)^{\langle\sigma\tau,\sigma^3\tau\rangle}=k(A,B)(u,v)$ 
where $A=\alpha^2$, $B=\beta^2$. 
The action of $\sigma$ on $k(A,B)(u,v)$ is given by 
\begin{align*}
\sigma: \ A\mapsto B, \ B\mapsto A, \ u\mapsto v, \ v\mapsto u.
\end{align*}
Define 
\[ 
a:=A+B, \ b:=(A-B)^2, \ U:=u+v, \ V:=(A-B)(u-v). 
\] 
Then $a, b\in K^{\langle\sigma, \tau \rangle}=k$ and 
$k(A,B)(u,v)^{\langle\sigma\rangle}=k(a,b)(U,V)=k(U,V)$. 
Hence $K(x,y)^G=k(U,V)$ is $k$-rational. \\ 

(ii) Suppose $(\alpha^2,c)_{2,k(\alpha^2)}\neq 0$. 
We have $K(x,y)^G\otimes_k F=K(x,y)^{\langle\sigma\tau,\sigma^3\tau\rangle}$ 
and the actions of $\sigma\tau$ and $\sigma^3\tau$ 
on $K(x,y)=k(\alpha,\beta)(x,y)$ are given by 
\begin{align*}
\sigma\tau &: \alpha\mapsto-\alpha,\ \beta\mapsto\beta,\ 
x\mapsto \frac{c}{x}, \ y\mapsto y,\\
\sigma^3\tau &: \alpha\mapsto\alpha,\ \beta\mapsto-\beta,\ 
x\mapsto x, \ y\mapsto \frac{c}{y}. 
\end{align*}
By Case 1 of Section \ref{ss37}, $K(x,y)^G\otimes_k F$ is not $F$-unirational 
where $F=K^{\langle\sigma\tau,\sigma^3\tau\rangle}=k(\alpha^2,\beta^2)=k(\alpha^2)$ 
with $[F:k]=2$. 
Hence $k(x,y)^G$ is not $k$-unirational. 
(Note that $(\alpha^2,c)_{2,k(\alpha^2)}\neq 0$ if and only if 
$(\beta^2,c)_{2,k(\beta^2)}\neq 0$.)\\

\textbf{Case 1-2:} $H=\{1\}$ and $\ep=-1$. Define
\begin{align*}
x^\prime:=\beta x,\quad y^\prime:=-\alpha y.
\end{align*}
Then the action of $G$ on $K(x,y)=K(x^\prime,y^\prime)$ is given by
\begin{align*}
\sigma &: \alpha\mapsto\beta,\ \beta\mapsto-\alpha,\ 
x^\prime\mapsto y^\prime, \ y^\prime\mapsto \frac{-\beta^2c}{x^\prime},\\
\sigma\tau &: \alpha\mapsto-\alpha,\ \beta\mapsto\beta,\ 
x^\prime\mapsto \frac{-\beta^2c}{x^\prime}, \ y^\prime\mapsto y^\prime,\\
\sigma^3\tau &: \alpha\mapsto\alpha,\ \beta\mapsto-\beta,\
x^\prime\mapsto x^\prime, \ y^\prime\mapsto \frac{-\alpha^2c}{y^\prime}.
\end{align*} 

(i) Suppose $(\alpha^2,-\beta^2c)_{2,k(\alpha^2)}=0$. 
Then there exists 
$\gamma_0\in k(\alpha)$ such that 
$-\beta^2c=
N_{k(\alpha)/k(\alpha^2)}(\gamma_0)=\gamma_0\gamma_1$ 
where $\gamma_1=\sigma\tau(\gamma_0)$.  
Define 
\begin{align*}
X:=\frac{x^\prime}{\gamma_0},\quad 
Y:=\frac{y^\prime}{\sigma(\gamma_0)}.
\end{align*}
Then $K(x,y)=K(X,Y)$ and the actions of 
$\sigma$, $\sigma\tau$ and $\sigma^3\tau$ on $K(X,Y)$ are given the same as in Case 1-1. 
Hence $K(x,y)^G=K(X,Y)^{\langle\sigma,\sigma\tau\rangle}$ is $k$-rational.\\ 

(ii) Suppose $(\alpha^2,-\beta^2c)_{2,k(\alpha^2)}\neq 0$. 
As in the case $\ep=1$, 
we have $K(x,y)^G\otimes_k F=K(x,y)^{\langle\sigma\tau,\sigma^3\tau\rangle}$ 
and the actions of $\sigma\tau$ and $\sigma^3\tau$ 
on $K(x,y)=k(\alpha,\beta)(x,y)$ are given by 
\begin{align*}
\sigma\tau &: \alpha\mapsto-\alpha,\ \beta\mapsto\beta,\ 
x^\prime\mapsto \frac{-\beta^2c}{x^\prime}, \ y^\prime\mapsto y^\prime,\\
\sigma^3\tau &: \alpha\mapsto\alpha,\ \beta\mapsto-\beta,\
x^\prime\mapsto x^\prime, \ y^\prime\mapsto \frac{-\alpha^2c}{y^\prime}.
\end{align*}
By Case 1 of Section \ref{ss37}, $K(x,y)^G\otimes_k F$ is not $F$-unirational 
where $F=K^{\langle\sigma\tau,\sigma^3\tau\rangle}=k(\alpha^2,\beta^2)
=k(\alpha^2)=k(\beta^2)$ 
with $[F:k]=2$. 
Hence $k(x,y)^G$ is not $k$-unirational. 
(Note that $(\alpha^2,-\beta^2c)_{2,k(\alpha^2)}$ $\neq 0$ if and only if 
$(\beta^2,-\alpha^2c)_{2,k(\beta^2)}\neq 0$.)

\begin{remark}\label{r3.4}
As in Section \ref{ss35}, 
one of the referees told that the case $G=D_4$ with $H=\{1\}$ corresponds to 
a del Pezzo surface $S$ of degree $8$ and it is well-known that 
$S$ is $k$-rational if and only if $S_F=S\otimes_k F$ is $F$-rational 
where $F=K^{\langle\sigma\tau,\sigma^3\tau\rangle}$ with $[F:k]=2$. 
It is also interesting to consider a more geometric proof 
and compare it with recent related papers 
Shramov and Vologodsky \cite[Section 7]{SV} and Trepalin \cite{Tr2}. 
\end{remark}

\noindent
\textbf{Case 2:} $H=\langle\sigma^2\rangle=\langle -I \rangle$.
We have $K=k(\sqrt{a},\sqrt{b})$ for some $a,b\in k$. 
The group $G$ acts on $K(x,y)$ by
\begin{align*}
\sigma &: \sqrt{a}\mapsto -\sqrt{a}, \ \sqrt{b}\mapsto \sqrt{b},
\ x\mapsto y, \ y\mapsto \frac{c}{x},\\
\tau &: \sqrt{a}\mapsto \sqrt{a}, \ \sqrt{b}\mapsto -\sqrt{b},\ 
x\mapsto \ep y, \ y\mapsto \ep x\ \ (\ep=\pm 1).
\end{align*}
By Lemma \ref{l2.1},
we have $K(x,y)^{\langle -I \rangle}=K(X,Y)$ where
\[
X=\frac{xy+c}{x+y},\quad 
Y=\frac{xy-c}{x-y}
\]
and the actions of $\sigma$ and $\tau$ on $K(X,Y)$ are given by
\begin{align*}
\sigma &: \sqrt{a}\mapsto -\sqrt{a}, \ \sqrt{b}\mapsto \sqrt{b}, \
X\mapsto \frac{c}{X}, \ Y\mapsto -\frac{c}{Y},  \\
\tau &: \sqrt{a}\mapsto \sqrt{a}, \ \sqrt{b}\mapsto -\sqrt{b},\ 
X\mapsto \ep X, \ Y\mapsto -\ep Y.
\end{align*}\vspace*{0mm}

\noindent
\textbf{Case 2-1:} $H=\langle\sigma^2\rangle=\langle -I \rangle$ and $\ep=1$.
Define
\[
S:=X,\quad T:=\sqrt{b}Y.
\]
Then $K(X,Y)^{\langle \tau \rangle}=k(\sqrt{a})(S,T)$
and the action of $\sigma$ on $k(\sqrt{a})(S,T)$ is given by
\begin{align*}
\sigma &: \sqrt{a}\mapsto -\sqrt{a},\ S\mapsto \frac{c}{S},\
T\mapsto -\frac{bc}{T}.
\end{align*}
It follows from Theorem \ref{t2.11} that
$K(x,y)^G$ is $k$-rational if and only if 
$K(x,y)^G$ is $k$-unirational if and only if 
$(a,c)_{2,k}=0$ and $(a,-bc)_{2,k}=0$.\\

\noindent
\textbf{Case 2-2:} $H=\langle\sigma^2\rangle=\langle -I \rangle$ and $\ep=-1$.
Define
\[
S:=\sqrt{b}X,\quad T:=Y.
\]
Then $K(X,Y)^{\langle \tau \rangle}=k(\sqrt{a})(S,T)$
and the action of $\sigma$ on $k(\sqrt{a})(S,T)$ is given by
\begin{align*}
\sigma &: \sqrt{a}\mapsto -\sqrt{a},\ S\mapsto \frac{bc}{S},\
T\mapsto -\frac{c}{T}.
\end{align*}
It follows from Theorem \ref{t2.11} that
$K(x,y)^G$ is $k$-rational if and only if 
$K(x,y)^G$ is $k$-unirational if and only if 
$(a,bc)_{2,k}=0$ and $(a,-c)_{2,k}=0$.\\

\noindent
\textbf{Case 3:} $H=\langle -I, \tau\sigma \rangle=\langle -I, \lambda\rangle$.
We have $K=k(\sqrt{a})$ for some $a\in k$. 
The group $G$ acts on $K(x,y)$ by
\begin{align*}
\sigma &: \sqrt{a}\mapsto -\sqrt{a},\ x\mapsto y,
\ y\mapsto \frac{c}{x},\\
\tau &: \sqrt{a}\mapsto -\sqrt{a}, \ 
x\mapsto \ep y, \ y\mapsto \ep x\ \ (\ep=\pm 1).
\end{align*}
By Lemma \ref{l2.1},
we have $K(x,y)^{\langle -I \rangle}=K(X,Y)$ where
\[
X=\frac{xy+c}{x+y},\quad 
Y=\frac{xy-c}{x-y}
\]
and the actions of $\sigma$ and $\tau$ on $K(X,Y)$ are given by
\begin{align*}
\sigma &: \sqrt{a}\mapsto -\sqrt{a},\  X\mapsto \frac{c}{X},
\ Y\mapsto -\frac{c}{Y},  \\
\tau &: \sqrt{a}\mapsto -\sqrt{a}, \ X\mapsto \ep X, \ Y\mapsto -\ep Y,\\
\tau\sigma &: \sqrt{a}\mapsto\sqrt{a}, \ X\mapsto \frac{\ep c}{X},
\ Y\mapsto \frac{\ep c}{Y}.
\end{align*}
Define
\[
S:=\frac{XY+\ep c}{X+Y},\quad T:=\frac{XY-\ep c}{X-Y}.
\]
Then, by Lemma \ref{l2.1} again, we have 
$K(X,Y)^{\langle \tau\sigma \rangle}=K(S,T)$ and 
\begin{align*}
\sigma=\tau &: \sqrt{a}\mapsto -\sqrt{a},\
S\mapsto -\ep T,\ T\mapsto -\ep S.
\end{align*}
Hence $K(x,y)^G=k(S-\ep T,\sqrt{a}(S+\ep T))$ is $k$-rational.\\

\noindent
\textbf{Case 4:} $H=\langle -I, \tau \rangle$.
We have $K=k(\sqrt{a})$ for some $a\in k$. 
The group $G$ acts on $K(x,y)$ by
\begin{align*}
\sigma &: \sqrt{a}\mapsto -\sqrt{a}, \  x\mapsto y, \
y\mapsto \frac{c}{x},\\
\tau &: \sqrt{a}\mapsto \sqrt{a},\ 
x\mapsto \ep y, \ y\mapsto \ep x\ \ (\ep=\pm 1).
\end{align*}
By Lemma \ref{l2.1},
we have $K(x,y)^{\langle -I \rangle}=K(X,Y)$ where
\[
X=\frac{xy+c}{x+y},\quad 
Y=\frac{xy-c}{x-y}.
\]
Then the actions of $\sigma$ and $\tau$ on $K(X,Y)$ are given by
\begin{align*}
\sigma &: \sqrt{a}\mapsto -\sqrt{a},\  X\mapsto \frac{c}{X},
\ Y\mapsto -\frac{c}{Y},  \\
\tau &: \sqrt{a}\mapsto \sqrt{a}, \ X\mapsto \ep X, \ Y\mapsto -\ep Y.
\end{align*}\vspace*{0mm}

\noindent
\textbf{Case 4-1:} $H=\langle -I, \tau \rangle$ and $\ep=1$.
Define
\[
S:=X,\quad T:=\frac{Y^2}{c}.
\]
Then $K(X,Y)^{\langle \tau \rangle}=K(S,T)$
and the action of $\sigma$ on $K(S,T)$ is given by
\begin{align*}
\sigma &: \sqrt{a}\mapsto -\sqrt{a}, \ S\mapsto \frac{c}{S},
\ T\mapsto \frac{1}{T}.
\end{align*}
It follows from Theorem \ref{t2.11} 
(or Theorem \ref{t2.6} after $T^\prime:=\sqrt{a}\,\frac{T+1}{T-1}$) that 
$K(x,y)^G$ is $k$-rational if and only if 
$K(x,y)^G$ is $k$-unirational if and only if 
$(a,c)_{2,k}=0$.\\

\noindent
\textbf{Case 4-2:} $H=\langle -I, \tau \rangle$ and $\ep=-1$.
Define
\[
S:=\frac{X^2}{c},\quad T:=Y.
\]
Then $K(X,Y)^{\langle \tau \rangle}=K(S,T)$
and the action of $\sigma$ on $K(S,T)$ is given by
\begin{align*}
\sigma &: \sqrt{a}\mapsto -\sqrt{a}, \ S\mapsto \frac{1}{S},
\ T\mapsto -\frac{c}{T}.
\end{align*}
It follows from Theorem \ref{t2.11} 
(or Theorem \ref{t2.6} after $S^\prime:=\sqrt{a}\,\frac{S+1}{S-1}$) that 
$K(x,y)^G$ is $k$-rational if and only if 
$K(x,y)^G$ is $k$-unirational if and only if 
$(a,-c)_{2,k}=0$.\\

\noindent
\textbf{Case 5:} $H=\langle \sigma  \rangle$.
We have $K=k(\sqrt{a})$ for some $a\in k$. 
The group $G$ acts on $K(x,y)$ by
\begin{align*}
\sigma &: \sqrt{a}\mapsto \sqrt{a}, \  x\mapsto y,
\ y\mapsto \frac{c}{x}, \\
\tau &: \sqrt{a}\mapsto -\sqrt{a},\ 
x\mapsto \ep y, \ y\mapsto \ep x\ \ (\ep=\pm 1).
\end{align*}
By Lemma \ref{l2.1}, we have
$K(x,y)^{\langle -I \rangle}=K(X,Y)$ where
\[
X=\frac{xy+c}{x+y},\quad 
Y=\frac{xy-c}{x-y}.
\]
Then the actions of $\sigma$ and $\tau$ on $K(X,Y)$ are given by
\begin{align*}
\sigma &: \sqrt{a}\mapsto \sqrt{a},\  X\mapsto \frac{c}{X},
\ Y\mapsto -\frac{c}{Y},  \\
\tau &: \sqrt{a}\mapsto -\sqrt{a}, \ X\mapsto \ep X, \ Y\mapsto -\ep Y.
\end{align*}\vspace*{0mm}

\noindent
\textbf{Case 5-1:} $H=\langle \sigma  \rangle$ and $\ep=1$.
Define $S:=X$, $T:=\sqrt{a}Y$.
Then we have $K(X,Y)^{\langle \tau \rangle}=k(S,T)$ and
\begin{align*}
\sigma &:  \ S\mapsto \frac{c}{S}, \ T\mapsto -\frac{ac}{T}.
\end{align*}
By Lemma \ref{l2.2}, $K(x,y)^G$ is $k$-rational.\\

\noindent
\textbf{Case 5-2:} $H=\langle \sigma  \rangle$ and $\ep=-1$.
Define $S:=\sqrt{a}X$, $T:=Y$.
Then we have $K(X,Y)^{\langle \tau \rangle}=k(S,T)$ and
\begin{align*}
\sigma &:  \ S\mapsto \frac{ac}{S}, \ T\mapsto -\frac{c}{T}.
\end{align*}
By Lemma \ref{l2.2}, $K(x,y)^G$ is $k$-rational.

\subsection{The case of $G=D_6=\langle \rho, \tau \rangle$}\label{ss312}


The action of $G$ on $K(x,y)$ is given by
\begin{align*}
\rho &: x\mapsto  bxy, \ y\mapsto \frac{c}{x},\\
\tau &: x\mapsto dy, \ y\mapsto ex\quad (b,c,d,e\in k\setminus\{0\}).
\end{align*}
By replacing $dy$ by $y$, $b/d$ by $b$ and $cd$ by $c$, 
we may assume that $d=e=1$.
By the equality $\tau\rho\tau^{-1}=\rho^{-1}$,
we have $c=1/b^2$.
Define $X:=bx$, $Y:=by$.
Then $K(x,y)=K(X,Y)$ and
the action of $G$ on $K(X,Y)$ is given by
\begin{align*}
\rho &: X\mapsto  XY, \ Y\mapsto \frac{1}{X},\\
\tau &: X\mapsto Y, \ Y\mapsto X.
\end{align*}
Hence the problem can be reduced to the purely quasi-monomial case.
By Theorem \ref{thHKY14},
$K(x,y)^G$ is $k$-rational for each of the proper normal subgroups
$H=\{1\}$, $\langle\rho^3\rangle=\langle -I \rangle$, $\langle \rho^2 \rangle$, $\langle \rho \rangle$, 
$\langle \rho^2, \tau \rangle, \langle \rho^2, -\tau \rangle$ of $G$.\\

{\it Proof of Corollary \ref{cor1.8}.} 
Because we have $k_\alpha(M)^G=k_\alpha(M_1)(M_2)^G$, 
it follows from Theorem \ref{thmain} that $k_\alpha(M)^G$ 
is rational over $k_\alpha(M_1)^G$. 
{\rm (1)} If ${\rm rank}_\bZ\, M_1=1$ or $2$, 
then $k_\alpha(M_1)^G$ is $k$-rational (see Theorem \ref{thHaj8387}). 
Hence $k_\alpha(M)^G$ is $k$-rational. 
{\rm (2)} If ${\rm rank}_\bZ\, M_1=3$, then 
the $k$-rationality of $k_\alpha(M_1)^G$ follows from \cite{HKiY} 
except for the case $G=D_6$ and the action of $G$ on $k_\alpha(M_1)$ 
is given as $G_{3,1,1}=\langle\tau_1,\lambda_1\rangle\simeq V_4$. 
Hence we also find that $k_\alpha(M)^G$ is $k$-rational.\qed

\section*{Declarations}
Conflict of Interest. The authors declare no competing interests. 

%

\end{document}